\documentclass[11pt,a4paper,final]{article}

\usepackage{amsmath}
\usepackage{amssymb}
\usepackage{amscd}
\usepackage{url}

\usepackage{graphicx}
\usepackage{color}
\usepackage{hyperref}
\hypersetup{bookmarks=true}
\hypersetup{colorlinks=false}

\usepackage[a4paper]{geometry}
\usepackage{amsthm}
\theoremstyle{plain}
\newtheorem{Thm}{Theorem}[section]
\newtheorem{Prop}[Thm]{Proposition}
\newtheorem{Lem}[Thm]{Lemma}

\theoremstyle{definition}

\theoremstyle{remark}
\newtheorem{Rmk}[Thm]{Remark}

\newtheorem*{proof*}{proof}

\newcommand{\A}{A}
\newcommand{\dm}[1]{d#1}
\newcommand{\E}{D}
\newcommand{\Ebr}{D}
\newcommand{\s}{s}
\newcommand{\Ca}{C_1(\varepsilon, \delta)}
\newcommand{\Cb}{C_2(\delta)}
\newcommand{\Cc}{C_3}
\makeatletter
 
 \@addtoreset{equation}{section}
\makeatother
\usepackage{mathtools}
\mathtoolsset{showonlyrefs=false}

\usepackage[color]{showkeys}
\renewcommand*\showkeyslabelformat[1]
{
  \parbox[t]
  {10mm}{\raggedright\normalfont\footnotesize\path{#1}}
}

\usepackage{fancyhdr}
\begin{document}

\allowdisplaybreaks[4]
\abovedisplayskip=4pt
\belowdisplayskip=4pt

\title{Large deviation principle for the intersection measure of Brownian motions on unbounded domains}
\author
{%
 Takahiro Mori%
 \thanks
 {
  Research Institute for Mathematical Sciences,
  Kyoto University, Kyoto, 606-8502, JAPAN.
  \url{tmori@kurims.kyoto-u.ac.jp}
 }
}
\date{}

\maketitle

\begin{abstract}
Consider
the intersection measure $\ell^{\mathrm{IS}}_t$ of
$p$ independent Brownian motions on $\mathbb{R}^d$.
%
In this article,
we prove the large deviation principle for
the normalized intersection measure
$t^{-p}\ell^{\mathrm{IS}}_t$ as $t\rightarrow \infty$,
before
exiting a (possibly unbounded) domain
$D\subset\mathbb{R}^d$ with smooth boundary.
This is an extension of 
[W. K\"onig and C. Mukherjee:
Communications on Pure and Applied Mathematics,
66(2):263--306, 2013]
which deals with the case $D$ is bounded.
Our essential contribution is to prove
the so-called super-exponential estimate
for the intersection measure of
killed Brownian motions on such $D$
by an application of the Chapman-Kolmogorov relation.
\end{abstract}





\section{Introduction}
\label{Sec_intro}

Analysis of the intersection of the Brownian paths
begins with the series of studies by
Dvoretzky, Erd\H{o}s, Kakutani and Taylor
\cite{MR0034972,MR0067402,MR0094855},
which give the following dichotomy:
for $p$ independent Brownian motions
$B^{(1)}, \ldots, B^{(p)}$ on $\mathbb{R}^d$,
the paths intersect,
i.e.,
$
 B^{(1)}(0, \infty)
\cap
 \cdots
\cap
 B^{(p)}(0, \infty)
\not=
 \varnothing
$
almost surely
if
$d-p(d-2)>0$,
and
do not intersect almost surely
if
$d-p(d-2)\leq 0$.
Motivated by physical problems
such as the configurations of interacting polymers,
two random measures that measure
the intensity of the intersections of the paths
have been introduced.
One is called the \emph{intersection local time} and
the other is called the \emph{intersection measure}.
A brief overview is given in Section \ref{Sec_relworks}.
Large deviation principles for this kind of measures
have recently been applied
by Mukherjee \cite{MR3716856}
to study a model of mutually interacting polymers, and
by Adams, Bru and K{\"o}nig \cite{MR2257650}
to prove the Gross-Pitaevskii formula
for the model of particles
with Dirac interaction potential.

In this paper,
we consider the \emph{(mutual) intersection measure}
introduced by K\"onig and Mukherjee \cite{MR2999298},
which is formally written as

\noindent
\begin{equation}
\label{eq_ISmeas}
 \ell^{\mathrm{IS}}_t(A)
=
 \int_A
 \biggl[
  \int_{[0, t]^p}
   \prod_{i=1} ^p
   \delta_x(B^{(i)}(s_i) )
  ds_1\cdots ds_p
 \biggr]
 dx
\quad
 \text{for }
 A\subset \mathbb{R}^d
 \text{ Borel}
\end{equation}
under the regime $d-p(d-2)>0$,
where $\delta_x$ is the Dirac measure at $x$
(see
Section \ref{Sec_results} for a precise definition).
Here and in the following,
the superscript ``IS'' means ``InterSection''.
A contribution of this paper is
the Donsker-Varadhan type large deviation principle
for the intersection measure
$\ell^{\mathrm{IS}}_t(dx)$ as $t\rightarrow \infty$,
before exiting an unbounded domain
$D\subset \mathbb{R}^d$ with smooth boundary.
This is roughly written as

\noindent
\begin{align}
\label{eq_LDPheu}
 \mathbb{P}
 \Bigl(
  (
   t^{-p}\ell^{\mathrm{IS}}_t
  ;
   t^{-1}\ell^{(1)}_t
  ,
   \dots
  ,
   t^{-1}\ell^{(p)}_t
  )
 \approx
  \boldsymbol{\mu}
 ,
  t
 <
  \tau^{(1)}_D \wedge \cdots \wedge \tau^{(p)}_D
 \Bigr)
\approx
 \exp
 \biggl\{
  -t
  \sum_{i=1} ^p
  \frac{1}{2}
  \int_D
   |\nabla \psi^{(i)}|^2
  dx
 \biggr\}
\end{align}

\noindent
as $t\rightarrow \infty$
(see Theorem \ref{Thm_LDPunbdd} for a precise meaning),
where
$\ell^{(i)}_t$ and $\tau^{(i)}_D$ are the
occupation measure and the exit time from $D$ of
the process $B^{(i)}$ respectively,
and
$
 \boldsymbol{\mu}
=
 (
  \mu
 ;
  \mu^{(1)}
 ,
  \ldots
 ,
  \mu^{(p)}
 )
$
is a tuple of a Radon measure
and $p$ probability measures on $D$
of the form
$
 \frac{d\mu}{dx}
=
 \prod_{i=1} ^p
 \frac{d\mu^{(i)}}{dx}
$
and
$
 \psi^{(i)}
=
 \sqrt{\frac{d\mu^{(i)}}{dx}}
\in
 W^{1, 2}_0(D)
$,
the Sobolev space with zero boundary values.
Previously,
K\"onig and Mukherjee \cite{MR2999298} showed
such large deviation result when $D$ is bounded.

This paper is organized as follows.
In Section \ref{Sec_results},
we recall the definition of
the intersection measure and state our main result
(Theorem \ref{Thm_LDPunbdd}),
the large deviation principle
of the intersection measure.
Section \ref{Sec_relworks}
summarizes earlier works related to our main result.
Section \ref{Sec_outline} is the outline
of the proof of the main result.
In this section,
we state the super-exponential estimate
(Theorem \ref{Thm_superexp}),
a key theorem to prove the main result.
We prove this
in the following Section \ref{Sec_superexp}.
In Section \ref{Sec_lowerLDP} and \ref{Sec_upperLDP},
we prove the large deviation lower and upper bound,
respectively.
Finally in Section \ref{Sec_stable},
we discuss an extension of the main theorems
from Brownian motions to other processes such as
the stable processes.


\subsection{Settings and main results}
\label{Sec_results}

Suppose
$D\subset \mathbb{R}^d$ be
a (possibly unbounded) domain with smooth boundary.
Let
$\partial$ be a point added to $D$ so that
$
 D_\partial
:=
 D\cup \{\partial\}
$
is the one-point compactification of $D$.
A killed Brownian motion $X$ in $D$ is the process
given by
\begin{equation*}
 X_t
=
 \begin{cases}
  B_t,      & t< \tau_D,
 \\
  \partial, & t\geq \tau_D,
 \end{cases}
\end{equation*}
where
$B$ is a Brownian motion on $\mathbb{R}^d$
and
$\tau_D = \inf\{t>0: B_t \not\in D\}$
is the exit time of $B$ from $D$.
We write
the continuous transition density function
and the $\alpha$-order resolvent density function
of $X$ by $p_t(x, y)$ and $r_\alpha(x, y)$,
respectively.

Set
the {\it ball average kernel} $q_\varepsilon(x, y)$ by
\begin{equation*}
 q_\varepsilon(x, y)
:=
 \frac{1}{|B(x, \varepsilon)|}
 1_{B(x, \varepsilon)}(y)
,
 \quad
 \varepsilon > 0
,
 x, y\in\mathbb{R}^d
\end{equation*}

\noindent
and
the {\it ball average operator} $T_\varepsilon$ by
\begin{equation*}
 T_\varepsilon f(x)
=
 \int_{\mathbb{R}^d}
  q_\varepsilon(x, y)
  f(y)
 dy
,
\quad
 \varepsilon > 0
,
 x\in \mathbb{R}^d
,
 f\in \mathcal{B}_b(\mathbb{R}^d)
,
\end{equation*}

\noindent
where
$B(x, \varepsilon)$ is the open ball
$
 \{
  y\in \mathbb{R}^d
 ;
  |x-y|<\varepsilon
 \}
$
and
$\mathcal{B}_b(\mathbb{R}^d)$ is the set
of bounded Borel functions on $\mathbb{R}^d$.
We note that
$T_\varepsilon$ is $L^r(\mathbb{R}^d)$-contractive
and strongly continuous as $\varepsilon \rightarrow 0$
for any $r\geq 1$.

Suppose that
$p\geq 2$ is an integer with $d-p(d-2)>0$.
Let
$X^{(1)}, \ldots, X^{(p)}$
be independent killed Brownian motions in $D$.
Throughout
this article, we fix their initial points
$x^{(1)}_0, \ldots, x^{(p)}_0 \in D$.
We write
$\tau^{(1)}_D, \ldots, \tau_D^{(p)}$
as their exit times from $D$, respectively.
For each $\varepsilon >0$,
we define
the {\it approximated (mutual) intersection measure}
$
 \ell^{\mathrm{IS}}_{t, \varepsilon}
$
of $X^{(1)}, \ldots, X^{(p)}$ up to time $t$ by
\begin{equation*}
 \langle
 \ell^{\mathrm{IS}}_{t, \varepsilon}
 ,
  f
 \rangle
=
 \int_{D}
  f(x)
  \biggl[
   \prod_{i=1} ^p
   \int_0 ^t
    q_\varepsilon(x, X^{(i)}_s)
   ds
  \biggr]
 \dm{x}
\end{equation*}
for $f\in \mathcal{B}_b(D)$,
with the convention
$
 q_\varepsilon(x, X^{(i)}_s)
=
 0
$
when $s\geq \tau^{(i)}_D$.
It is well known that for each $t>0$,
there exists a random measure
$\ell^{\mathrm{IS}}_t$
such that
$
 \ell^{\mathrm{IS}}_{t, \varepsilon}
$
converges vaguely to
$\ell^{\mathrm{IS}}_t$
in
$\mathcal{M}(D)$
and that
\begin{equation*}
 \lim_{\varepsilon \rightarrow 0}
 \mathbb{E}
 \bigl[
  |
   \langle
    f, \ell^{\mathrm{IS}}_{t, \varepsilon}
   \rangle
  -
   \langle
    f, \ell^{\mathrm{IS}}_{t}
   \rangle
  |^k
 \bigr]
=
 0
\quad
 \text{for any integer $k\geq 1$ and }
 f\in C_K(D)
,
\end{equation*}
where
$\mathcal{M}(D)$ is the set of Radon measures on $D$
equipped with the vague topology $\tau_v$
and $C_K(D)$ is the set of continuous functions
on $D$ with compact support.
The limit $\ell^{\mathrm{IS}}_t$ is called
the
{\it
 (mutual) intersection measure
 of $X^{(1)}, \ldots, X^{(p)}$
 up to time t%
}.
For detail,
see \cite{MR2999298} or the author's previous paper
\cite{Mor20} for example.

\vspace{1eM}

Before stating our main results,
we recall the definition of large deviation principle.
Usually
the large deviation principle is
defined for families of probability measures,
but in this paper
we define for families of sub-probability measures.
Note that
most of the basic properties of
large deviation principle
(e.g.,
contraction principle)
also hold in this case.

Let $\mathcal{X}$ be a topological space.
A function
$I: \mathcal{X} \rightarrow [0, +\infty]$
is called a {\it rate function}
(resp. {\it good rate function})
if
for any $\alpha \geq 0$
the level set
$\{x\in \mathcal{X}: I(x)\leq \alpha\}$
is closed (resp. compact) in $\mathcal{X}$.
We say that the family of
sub-probability measures $\{\mathcal{P}_t\}_{t>0}$
on $\mathcal{X}$
satisfies the {\it (full) large deviation principle}
(LDP in abbreviation) as $t\rightarrow \infty$
with rate function $I$
if
\begin{equation}
\label{eq_LDPupper}
 \limsup_{t\rightarrow\infty}
 \frac{1}{t}
 \log
 \mathcal{P}_t(F)
\leq
 -\inf_{\mu \in F} I(\mu)
\quad
 \text{for all closed set }
 F\subset \mathcal{X}
\end{equation}
and
\begin{equation}
\label{eq_LDPlower}
 \liminf_{t\rightarrow\infty}
 \frac{1}{t}
 \log
 \mathcal{P}_t(G)
\geq
 -\inf_{\mu \in G} I(\mu)
\quad
 \text{for all open set }
 G\subset \mathcal{X}
.
\end{equation}
The condition \eqref{eq_LDPupper}
and \eqref{eq_LDPlower}
are referred to
as the LDP upper and lower bounds,
respectively.
We also define the {\it weak LDP}
by replacing
all closed sets
with
all compact sets
in the definition of LDP upper bound
\eqref{eq_LDPupper}.

Our main result is the weak LDP
for the intersection measure.
This is a natural formulation for unbounded domains
since (full) LDP fails to hold even for a single
empirical measure when $D=\mathbb{R}^d$.
When the domain $D$ is unbounded,
we need to consider the case
that some mass of the (normalized) occupation measure
of a Brownian motion escapes to infinity.
Hence,
for the occupation measure,
it is natural to consider the LDP
on the space $\mathcal{M}_1(D_\partial)$,
the set of probability measures
on $D_\partial$ equipped with the weak topology,
even when $D$ is the whole space $\mathbb{R}^d$.
We
can see that this is equivalent to
the set of sub-probability measures
$\mathcal{M}_{\leq 1}(D)$
equipped with the vague topology.

Define the function
$
 \mathbf{I}
:
 \mathcal{M}(D) \times (\mathcal{M}_{1}(D))^p
\rightarrow
 [0, +\infty]
$
by
\begin{align}
\notag
&
 \mathbf{I}(\mu ; \mu^{(1)}, \ldots, \mu^{(p)})
\\
\label{eq_bfI}
&:=
 \begin{cases}
  \displaystyle
  \frac{1}{2}
  \sum_{i=1} ^p
  \int_{D}
   |\nabla \psi^{(i)}|^2
  dx
 ,
 &
  \text{if }
  \psi^{(i)}
 =
  \displaystyle
  \sqrt{\frac{d\mu^{(i)}}{dx}}
 \in
  W^{1, 2}_0(D)
 \text{ and }
  \displaystyle
  \prod_{i=1} ^p
  \frac{d\mu^{(i)}}{dx}
 =
  \frac{d\mu}{dx}
 ,
 \\
  \infty
 ,
 &
  \text{otherwise}
 \end{cases}
\end{align}
for
$
 (\mu ; \mu^{(1)}, \ldots, \mu^{(p)})
\in
 \mathcal{M}(D) \times (\mathcal{M}_{1}(D))^p
$,
where
$\mathcal{M}_1(D)$ is the set of
probability measures on $D$
equipped with the weak topology $\tau_w$.
Write
the occupation measure $\ell^{(i)}_t$ of $X^{(i)}$
up to $t$ by
$
 \langle f, \ell^{(i)}_t \rangle
=
 \int_0 ^t
  f(X^{(i)}_s)
 ds
$
for bounded Borel functions $f$ on $D$.
The following weak LDP is our main result.

\begin{Thm}
\label{Thm_LDPunbdd}
On the space
$
 (\mathcal{M}(D), \tau_v)
\times
 (\mathcal{M}_{1}(D), \tau_w)^p
$,
the law of the tuple
$
 (
  t^{-p}\ell^{\mathrm{IS}}_{t}
 ;
  t^{-1}\ell^{(1)}_t
 ,
  \ldots
 ,
  t^{-1}\ell^{(p)}_t
 )
$
satisfies the weak LDP
as $t\rightarrow \infty$
under
$
 \mathbb{P}
 (
  \hspace{1mm}\cdot\hspace{1mm}
 ,
  t
 <
  \tau^{(1)}_D
 \wedge
  \cdots
 \wedge
  \tau^{(p)}_D
 )
$,
with the rate function
$
 \mathbf{I}
$.
\end{Thm}

\noindent
In fact
we will show a stronger result
(Theorem \ref{Thm_LDP}),
in which
the weak LDP upper bound is replaced by
the full LDP upper bound on the space
$
 (\mathcal{M}(D), \tau_v)
\times
 (\mathcal{M}_{\leq 1}(D), \tau_v)^p
$,
where
$\mathcal{M}_{\leq 1}(D)$ is the set of
sub-probability measures on $D$
equipped with the vague topology $\tau_v$.

\begin{Rmk}
The above result holds
not only for Brownian motion
but for more general processes.
In Remark \ref{Rmk_list} and
at the end of Section \ref{Sec_stable},
we list the conditions for the process we used in the proof.
\end{Rmk}

When
the domain $D$ is bounded,
there is no difference between weak and full LDP's.
Thus
we recover \cite[Theorem 1.1]{MR2999298}.

\begin{Prop}
\label{Prop_LDPbdd}
Suppose
$D\subset \mathbb{R}^d$
is a bounded domain with smooth boundary.
Then,
on the space
$
 (\mathcal{M}(D), \tau_v)
\times
 (\mathcal{M}_{1}(D), \tau_w)^p
$,
the law of the tuple
$
 (
  t^{-p}\ell^{\mathrm{IS}}_{t}
 ;
  t^{-1}\ell^{(1)}_t
 ,\allowbreak
  \ldots
 ,\allowbreak
  t^{-1}\ell^{(p)}_t
 )
$
satisfies the full LDP
as $t\rightarrow \infty$
under
$
 \mathbb{P}
 (
  \hspace{1mm}\cdot\hspace{1mm}
 ,
  t
 <
  \tau^{(1)}_D
 \wedge
  \cdots
 \wedge
  \tau^{(p)}_D
 )
$,
with the good rate function
$
 \mathbf{I}
$.
\end{Prop}


\subsection{Related works}
\label{Sec_relworks}

In this section,
we briefly review some works
which are related to this paper.
The interested reader
may refer, for example, to \cite{MR1229519, MR2584458}
for more information on intersection properties of
Brownian motions.

Let us first recall another random measure called the
\emph{intersection local time}
which is mentioned at the beginning.
This measure can be formally written as
\begin{equation}
\label{eq_ISLT}
 \alpha(A)
=
 \int_{\mathbb{R}^d}
 \biggl[
  \int_{A}
   \prod_{i=1} ^p
   \delta_x(B^{(i)}(s_i) )
  ds_1\cdots ds_p
 \biggr]
 dx
\quad
 \text{for }
 A\subset [0, \infty)^p
 \text{ Borel}
.
\end{equation}
The precise meaning and a construction of this measure
can be found in the work by
Geman, Horowitz and Rosen \cite{MR723731} and
Le Gall \cite{MR1229519}, for example.
It should be emphasized that
this measure $\alpha$ is different to
our $\ell^{\mathrm{IS}}$ since the former measures
\emph{when} the Brownian paths intersect,
while the latter measures
\emph{where} the Brownian paths intersect.
These two measures have the same total mass
$
 \alpha([0, t]^p)
=
 \ell^{\mathrm{IS}}_t(\mathbb{R}^d)
$,
but otherwise there seems to be no other direct relation.
In fact,
Geman, Horowitz and Rosen \cite{MR723731} constructed
a more general object formally written as
\begin{equation}
\label{eq_GHR}
 \alpha(x, A)
=
 \int_{A}
  \prod_{i=1} ^{p-1}
  \delta_{x_i}(B^{(i+1)}(s_{i+1}) - B^{(i)}(s_i) )
 ds_1\cdots ds_p
\quad
 \text{for }
 A\subset [0, \infty)^p
 \text{ Borel}
\end{equation}
for each $x = (x_1, \ldots, x_{p-1})\in(\mathbb{R}^d)^{p-1}$,
which is supported on the set
$
 \{
  (s_1, \ldots, s_p)
 \in
  [0, \infty)^p
 :
  B^{(i+1)}(s_{i+1})
 -
  B^{(i)}(s_{i})
 =
  x_i
 \text{ for all }i=1, \cdots, p-1
 \}
$.
The two measures in \eqref{eq_ISLT} and \eqref{eq_GHR}
are related as
$
 \alpha(ds_1\cdots ds_p)
=
 \alpha(0, ds_1\cdots ds_p)
$.
Note that
\eqref{eq_GHR} has the spatial variable $x$
but
its role is still different to that of $x$
in $\ell^{\mathrm{IS}}$.

Now let us turn to the earlier studies on
large deviations of Brownian intersections.
Most of the works are
about the total mass of the intersection measure.
K\"onig and M\"orters \cite{MR1944002} investigated
upper tail asymptotics of the random variable
$\ell^{\mathrm{IS}}_\infty(D)$
(this is equal to
$
 \alpha
 (
  [0, \tau^{(1)}_D]
 \times \cdots \times
  [0, \tau^{(p)}_D]
 )
$)
when $D$ is bounded.
Chen \cite{MR2094445}
studied the law of the iterated logarithm about
$
 \ell^{\mathrm{IS}}_t(\mathbb{R}^d)
$
(this equals $\alpha([0, t]^p)$)
as $t\rightarrow \infty$
and
showed that the asymptotics
of the logarithmic moment generating function
$
 \log
 \mathbb{E}
 \exp
 \bigl\{
  \theta
  \ell_t ^{\textrm{IS}}(\mathbb{R}^d)^{1/p}
 \bigr\}
$
as $t\rightarrow \infty$
can be represented as a variational formula.
Chen and Rosen \cite{MR2165257}
proved that similar results also hold
for $p$ independent $\alpha$-stable processes on
$\mathbb{R}^d$ with $0< \alpha < 2$.
%
More recently, Mukherjee \cite{MR3716856} established
a full large deviation principle
in the same spirit of \cite{MR3572328},
for the distribution of the \emph{orbits} of
the product of two occupation measures
$(t^{-1}\ell_t^{(1)}, t^{-1}\ell_t^{(2)})$
which are embedded in a larger space
equipped with a new topology
(making it the compactification of quotient space
of orbits of product measures),
and applied it to study large deviation estimates for
$\ell_t^{\mathrm{IS}}(\mathbb{R}^3)$.

It is the work by K\"onig and Mukherjee \cite{MR2999298}
that first considered (and established) the LDP
for the intersection measure $\ell^{\mathrm{IS}}_t$.
In that paper,
they used the eigenvalue expansion
of the transition density,
and hence their argument requires that
the domain $D$ is bounded.


\section{Proof outline}
\label{Sec_outline}

Our proof follows
the same three steps as in \cite{MR2999298}:

\begin{itemize}
\setlength{\itemsep}{0mm}
\item
LDP for the approximated intersection measure
$\ell^{\mathrm{IS}}_{t, \varepsilon}$,
\item
Convergence of the rate function as $\varepsilon\to 0$,
\item
Super-exponential estimate for
$\ell^{\mathrm{IS}}_{t, \varepsilon}- \ell^{\mathrm{IS}}_{t}$.
\end{itemize}

\noindent
The first two steps
do not require new ideas and well be done
in Sections \ref{Sec_lowerLDP} and \ref{Sec_upperLDP}.
The technical novelty of this paper lies in
the third step.
More precisely,
we show the following extension of
\cite[Proposition 2.3]{MR2999298}.
Let
$D\subset \mathbb{R}^d$ be a (possibly unbounded)
domain with smooth boundary.

\begin{Thm}
[Super-exponential estimate]
\label{Thm_superexp}
For each $f\in C_K(D)$
and $\varepsilon > 0$,
there exists positive constant
$C(\varepsilon)$,
which depends on $p$ and $f$
and is independent of $t$ and $k$,
such that
$
 \lim_{\varepsilon \rightarrow 0}
 C(\varepsilon)
=
 0
$
and

\noindent
\begin{equation}
\label{eq_superexp}
 \mathbb{E}
 [
  |
   \langle
    \ell^{\mathrm{IS}}_{t, \varepsilon}
   ,
    f
   \rangle
  -
   \langle
    \ell^{\mathrm{IS}}_{t}
   ,
    f
   \rangle
  |^k
 ]
\leq
 e^{pt}
 (k!)^p
 C(\varepsilon)^k
\quad
 \text{ for any } k\geq 1
 \text{ and } t>0
.
\end{equation}
\end{Thm}

\noindent
We prove this
in the following Section \ref{Sec_superexp}.

We note that
Theorem \ref{Thm_superexp} and the Markov inequality
imply that the random variables
$
 \{
  t^{-p}\ell^{\mathrm{IS}}_{t, \varepsilon}
 \}_{t, \varepsilon}
$
are exponentially good approximations of
$
 \{
  t^{-p}\ell^{\mathrm{IS}}_{t}
 \}_t
$,
that is, it holds that
\begin{equation*}
 \lim_{\varepsilon \rightarrow 0}
 \limsup_{t\rightarrow \infty}
 \frac{1}{t}
 \log
 \mathbb{P}
 \bigl(
  d
  (
   t^{-p}\ell^{\mathrm{IS}}_{t, \varepsilon}
  ,
   t^{-p}\ell^{\mathrm{IS}}_{t}
  )
 >
  \delta
 \bigr)
=
 -\infty
\end{equation*}
for every $\delta >0$,
where $d$ is a metric on $\mathcal{M}(D)$
associated with $\tau_v$.



Once
Theorem \ref{Thm_superexp}
(and hence the exponentially good approximation)
is proved,
we can deduce the following large deviation result
by a similar argument
as in the proof of \cite[Theorem 1.1]{MR2999298}.
Let
$\ell^{(1)}_t, \ldots, \ell^{(p)}_t$
be the occupation measure of $X^{(1)}, \ldots, X^{(p)}$
up to $t$, respectively.
We
extend the definition of $\mathbf{I}$ from
$
 \mathcal{M}(D) \times (\mathcal{M}_{1}(D))^p
$
to
$
 \mathcal{M}(D) \times (\mathcal{M}_{\leq 1}(D))^p
$
canonically
and write the extended function as
$
 \overline{\mathbf{I}}
$.

\begin{Thm}
\label{Thm_LDP}
\ \par
\begin{enumerate}
\setlength{\itemsep}{1mm}
\item[(i)]
On the space
$
 (\mathcal{M}(D), \tau_v)
\times
 (\mathcal{M}_{1}(D), \tau_w)^p
$,
the law of the tuple
$
 (
  t^{-p}\ell^{\mathrm{IS}}_{t}
 ;
  t^{-1}\ell^{(1)}_t
 ,\allowbreak
  \ldots
 ,
  t^{-1}\ell^{(p)}_t
 )
$
satisfies the LDP lower bound
as $t\rightarrow \infty$
under
$
 \mathbb{P}
 (
  \hspace{1mm}\cdot\hspace{1mm}
 ,
  t
 <
  \tau^{(1)}_D
 \wedge
  \cdots
 \wedge
  \tau^{(p)}_D
 )
$,
with the rate function
$
 \mathbf{I}
$.

\item[(ii)]
On the space
$
 (\mathcal{M}(D), \tau_v)
\times
 (\mathcal{M}_{\leq 1}(D), \tau_v)^p
$,
the law of the tuple
$
 (
  t^{-p}\ell^{\mathrm{IS}}_{t}
 ;
  t^{-1}\ell^{(1)}_t
 ,\allowbreak
  \ldots
 ,
  t^{-1}\ell^{(p)}_t
 )
$
satisfies the LDP upper bound
as $t\rightarrow \infty$
under
$
 \mathbb{P}
 (
  \hspace{1mm}\cdot\hspace{1mm}
 ,
  t
 <
  \tau^{(1)}_D
 \wedge
  \cdots
 \wedge
  \tau^{(p)}_D
 )
$,
with the good rate function
$
 \overline{\mathbf{I}}
$.

\end{enumerate}
\end{Thm}

We remark that
Theorem \ref{Thm_LDP} (i)
is exactly the same as the LDP lower bound of
Theorem \ref{Thm_LDPunbdd}.
Theorem \ref{Thm_LDP} (ii)
implies
the LDP upper bound of Theorem \ref{Thm_LDPunbdd}
because
all compact sets of
$
 (\mathcal{M}(D), \tau_v)
\times
 (\mathcal{M}_{1}(D), \tau_w)^p
$
are closed in
$
 (\mathcal{M}(D), \tau_v)
\times
 (\mathcal{M}_{\leq 1}(D), \tau_v)^p
$.
Therefore
Theorem \ref{Thm_LDPunbdd} is proved in this way.

As for (i),
unlike the LDP upper bound (ii),
the LDP lower bound on the space
$
 (\mathcal{M}(D), \tau_v)
\times
 (\mathcal{M}_{\leq 1}(D), \tau_v)^p
$
does not hold in general.
Indeed,
for the open set
$
 G
=
 \mathcal{M}(D)\times\mathcal{M}_{\leq 1}(D)^p
$
of
$
 (\mathcal{M}(D), \tau_v)
\times
 (\mathcal{M}_{\leq 1}(D), \tau_v)^p
$
we have
$
 -
 \inf_G
 \overline{\mathbf{I}}
=
 0
$.
On the other hand, we can find that the value
\begin{align*}
&
 \liminf_{t\rightarrow \infty}
 \frac{1}{t}
 \log
 \mathbb{P}
 \Bigl(
  (
   t^{-p}\ell^{\mathrm{IS}}_t
  ;
   t^{-1}\ell^{(1)}_t
  ,
   \dots
  ,
   t^{-1}\ell^{(p)}_t
  )
 \in
  G
 ,
  t
 <
  \tau^{(1)}_D \wedge \cdots \wedge \tau^{(p)}_D
 \Bigr)
\\
=&
 p
 \liminf_{t\rightarrow \infty}
 \frac{1}{t}
 \log
 \mathbb{P}
 (
  t
 <
  \tau_D^{(1)}
 )
\end{align*}
may be negative.


\section{Proof of Theorem \ref{Thm_superexp} : super-exponential estimate}
\label{Sec_superexp}

First, we heuristically state our idea
of the proof of the super-exponential estimate.
For simplicity,
we assume that
the processes $X^{(1)}, \ldots, X^{(p)}$ have
the same initial point $x_0\in D$
and we only consider the case
$k$ is an even integer and
$f=1_U$, the indicator function of
a relatively compact open subset $U$ of $D$.
Note that an analogy of the Le Gall's moment formula
\begin{align}
\label{eq_LGMF}
 \mathbb{E}
 [
  \langle
   \ell^{\mathrm{IS}}_{t}
  ,
   f
  \rangle^k
 ]
=
 \int_{{\E}^k}
  f(x_1) \cdots f(x_k)
  \prod_{i=1}^p
  \biggl(
   \sum_{\sigma\in \mathfrak{S}_k}
   H^{(i)}_t
   (x_{\sigma(1)}, \ldots, x_{\sigma(k)})
  \biggr)
 \dm{x_1}\cdots \dm{x_k}
\end{align}
holds for the intersection measure $\ell^{\mathrm{IS}}_t(dx)$,
where
\begin{equation*}
 H_t(x_1, \ldots, x_k)
:=
 \int_{[0, \infty)^k}
  1_{\{\sum_{j=1} ^k \s_j \leq t\}}
  h_{\boldsymbol{s}}(x_1, \ldots, x_p)
 d\s_1 \cdots d\s_k
\end{equation*}
and
\begin{equation*}
 h_{\boldsymbol{s}}(x_1, \ldots, x_p)
:=
 p_{s_1}(x_0, x_1)1_U(x_1)
\cdots
 p_{s_k}(x_{k-1}, x_k)1_U(x_k)
.
\end{equation*}
This type of moment formulae are firstly obtained
in \cite{MR902429}
for the intersection local time \eqref{eq_ISLT},
and the same method also works for the intersection measure
(see \cite[Lemma 5.1]{1805.07945} for example).
This formula and
a straight calculation give that
for sufficient small $\varepsilon>0$

\noindent
\begin{align*}
 \mathbb{E}
 [
  |
   \langle
              {\ell^{\mathrm{IS}}_{t, \varepsilon}}
   ,
    f
   \rangle
  -
   \langle
    \ell^{\mathrm{IS}}_{t}
   ,
    f
   \rangle
  |^k
 ]
\leq
 (k!)^p
 \bigl\|
  (T_\varepsilon  - \mathrm{id})^{\otimes k}
  H_t
 \bigl\|_{L^p(\E^k)} ^p
.
\end{align*}


\noindent
Then our goal is to estimate
the function
$
 (T_\varepsilon  - \mathrm{id})^{\otimes k}
 h_{\boldsymbol{s}}
$
with respect to the integral
$
 \int_{[0, \infty)^k}
 1_{\{\sum_{j=1} ^k \s_j \leq t\}}
$
$
 d\s_1 \cdots d\s_k
$
and then the norm
$
 \|
  \cdot
 \|_{L^p(\E^k)}
$.

Now,
fix small $\delta>0$ and focus on the regime
$s_1, \ldots, s_p \geq \delta$
where we need a new idea.
By
setting $u_j = \s_j -\delta$,
the Chapman-Kolmogorov relations

\noindent
\begin{align*}
 p_{u_1 + \delta}(x_0, x_1)
=&
 \int_{\E}
  p_{\frac{\delta}{2}+u_1}(x_0, z_1)
  p_{\frac{\delta}{2}    }(z_1, x_1)
 \dm{z_1}
,
\\
 p_{u_{j} + \delta}(x_{j-1}, x_j)
=&
 \int_{\E}
 \int_{\E}
  p_{\frac{\delta}{2}}(x_{j-1}, y_j)
  p_{u_j             }(y_j    , z_j)
  p_{\frac{\delta}{2}}(z_j    , x_j)
 \dm{y_j}
 \dm{z_j}
\quad
 2\leq j\leq k-1
,
\\
 p_{u_k + \delta}(x_{k-1}, x_k)
=&
 \int_{\E}
  p_{\frac{\delta}{2}    }(x_{k-1}, y_k)
  p_{\frac{\delta}{2}+u_k}(y_k    , x_k)
 \dm{y_k}
\end{align*}

\noindent
give that
\begin{align*}
&
 h_{\boldsymbol{s}}(x_1, \ldots, x_k)
\\
\begin{split}
=&
 \int_{D^{k-1}}
 d\boldsymbol{y}
 \int_{D^{k-1}}
 d\boldsymbol{z}
\\
&\hspace{5mm}
  p_{\frac{\delta}{2}+u_1}(x_0, z_1)
  p_{\frac{\delta}{2}    }(z_1, x_1)
 \biggl(
  \prod_{j=2} ^{k-1}
  p_{\frac{\delta}{2}}(x_{j-1}, y_j)
  p_{u_j             }(y_j    , z_j)
  p_{\frac{\delta}{2}}(z_j    , x_j)
 \biggr)
  p_{\frac{\delta}{2}    }(x_{k-1}, y_k)
  p_{\frac{\delta}{2}+u_k}(y_k    , x_k)
\end{split}
\\
\begin{split}
=&
 \int_{D^{k-1}}
 d\boldsymbol{y}
 \int_{D^{k-1}}
 d\boldsymbol{z}
 \hspace{2mm}
  p_{\frac{\delta}{2}+u_1}(x_0, z_1)
  \Bigl(
   p_{\frac{\delta}{2}    }(z_1, x_1)
   p_{\frac{\delta}{2}    }(x_1, y_2)
  \Bigr)
  p_{u_2             }(y_2    , z_2)
  \Bigl(
   p_{\frac{\delta}{2}    }(z_2, x_2)
   p_{\frac{\delta}{2}    }(x_2, y_3)
  \Bigr)
\\
 &\hspace{5mm}
  p_{u_3             }(y_3    , z_3)
 \cdots
  p_{u_{k-1}}(y_{k-1}, z_{k-1})
  \Bigl(
   p_{\frac{\delta}{2}}(z_{k-1}, x_{k-1})
   p_{\frac{\delta}{2}}(x_{k-1}, y_k)
  \Bigr)
  p_{\frac{\delta}{2}+u_k}(y_k    , x_k)
.
\end{split}
\end{align*}

\noindent
Here,
note that in the integrand
the variables $x_1, \ldots, x_k$ appear
in different functions.
This allows us to apply
$(T_\varepsilon - \mathrm{id})^{\otimes k}$
\emph{separately}.
Hence
all what we need are
estimates on the functions (of $x$)
\begin{equation*}
 (T_\varepsilon - \mathrm{id})
 \bigl[
  p_{\frac{\delta}{2}}(z_j  , \cdot  )
  p_{\frac{\delta}{2}}(\cdot, y_{j+1})
  1_{U}(\cdot)
 \bigr](x)
,
 \hspace{1ex}
 1\leq j\leq k-1
\quad
 \text{and}
\quad
 p_{\frac{\delta}{2}+u_k}(y_k    , x)
 1_U(x)
\end{equation*}
uniformly over
$y_2, \ldots, y_k$, $z_1, \ldots, z_{k-1}$ and $u_k$,
with respect to proper norms.

\begin{Rmk}
In the previous work \cite{MR2999298},
when the domain $D$ is bounded,
K\"onig and Mukherjee
proved the super-exponential estimated
by using the eigenvalue expansion

\noindent
\begin{align*}
 p_t(x, y)
=
 \sum_{n=1} ^\infty
 e^{-t\lambda_n}
 \psi_n(x)
 \psi_n(y)
\end{align*}
to \emph{separate} the variables
instead of the Chapman-Kolmogorov equation,
where
$
 0<\lambda_1\leq \lambda_2 \leq \cdots
$
and
$\psi_n\in W^{1, 2}_0(D)$
satisfies
$
 -\frac{1}{2}\triangle \psi_n
=
 \lambda_n \psi_n
$.
\end{Rmk}

\vspace{1eM}


We now prove Theorem \ref{Thm_superexp}.

\begin{proof}
[Proof of Theorem \ref{Thm_superexp}]
Fix $f\in C_K(\E)$
and take a relatively compact open set $U$ with
$\mathrm{supp}[f] \subset U$.
For
$\varepsilon > 0$ and $\delta > 0$, set
\begin{equation}
\label{eq_Cepdel}
 \Ca
:=
 \sup_{z\in \E}
 \biggl\{
  \int_{\E}
   \Bigl\|
    (T_\varepsilon - \mathrm{id})
    \bigl[
    p_{\frac{\delta}{2}}(z    , \cdot)
    p_{\frac{\delta}{2}}(\cdot, y    )
    1_{U}(\cdot)
    \bigr]
   \Bigr\|_{L^p(\E)}
  \dm{y}
 \biggr\}
\end{equation}

\noindent
and
\begin{equation}
\label{eq_tau}
 \Cb
:=
 \sup_{x\in \E}
 \biggl\{
  \int_{\E}
   \biggl(
    \int_0 ^\delta
     p_s(x, y)
    ds
   \biggr)^p
  \dm{y}
 \biggr\}^{1/p}
.
\end{equation}

\noindent
We also write the constant
\begin{equation}
\label{eq_gamma}
 \Cc
:=
 \sup_{x\in \E}
 \biggl\{
  \int_{\E}
   r_1(x, y)^p
  \dm{y}
 \biggr\}^{1/p}
.
\end{equation}

\noindent
We easily find that
$
 \lim_{\delta\downarrow 0}
 \Cb
=
 0
$
and
$
 \Cc
<
 \infty
$
because of the bound
$
 p_\delta(x, y)
\leq
 (2\pi)^{-d/2}
 \exp\{-|x-y|^2/2\delta\}
$
and the assumption $d-p(d-2) > 0$
(see equation (2.2.10) and Theorem A.1 of
\cite{MR2584458} for example).
In
Lemma \ref{Lem_Cepsdel} later,
we will see that
$
 \lim_{\varepsilon \rightarrow 0}
 \Ca
=
 0
$.
Therefore,
we can derive the conclusion \eqref{eq_superexp}
by the same argument as in \cite{Mor20},
as soon as we show the following:
for sufficiently small $\delta>0$ and
$\varepsilon>0$ with
$
 \Ca
+
 \Cb
<
 1
$,
it holds that

\noindent
\begin{equation}
\label{eq_superexp_2}
 \mathbb{E}
 [
  |
   \langle
    {\ell^{\mathrm{IS}}_{t, \varepsilon}}
   ,
    f
   \rangle
  -
   \langle
    \ell^{\mathrm{IS}}_{t}
   ,
    f
   \rangle
  |^k
 ]
\leq
 e^t
 (k!)^p
 \|f\|_\infty ^k
 \bigl\{
  16
  (\Cc + 1)
  (
   \Cb
  +
   \Ca
  )^{\frac{1}{6}}
 \bigr\}^{pk}
,
\end{equation}
for any $k\geq 1$ and $t>0$.

From now on,
we prove \eqref{eq_superexp_2}.
For
small $\varepsilon > 0$ such that
$
 d(\mathrm{supp}[f], \E\setminus U)
\geq
 \varepsilon
$,
we have

\noindent
\begin{align*}
 \langle
  {\ell^{\mathrm{IS}}_{t, \varepsilon}}
 ,
  f
 \rangle
=&
 \int_{\E}
  f(x)
  \biggl[
   \prod_{i=1} ^p
   \int_0 ^t
    q_\varepsilon(x, X^{(i)}_s)
   ds
  \biggr]
 \dm{x}
\\
=&
 \int_{\E}
  f(x)
  \biggl[
   \prod_{i=1} ^p
   \int_0 ^t
    q_\varepsilon(x, X^{(i)}_s)
    1_{U}(X^{(i)}_s)
   ds
  \biggr]
 \dm{x}
.
\end{align*}

\noindent
In the following,
we fix an even integer $k\geq 2$.
Set
\begin{equation*}
 H^{(i)}_t(x_1, \ldots, x_k)
:=
 \int_{[0, \infty)^k}
  1_{\bigl\{\sum_{j=1} ^k \s_j \leq t\bigr\}}
  \biggl[
   \int_{\E}
   \prod_{j=1} ^k
   p_{\s_j}(x_{j-1}, x_j)
   1_{U}(x_j)
   \nu^{(i)}(dx_0)
  \biggr]
 d\s_1 \cdots d\s_k
,
\end{equation*}
where
$
 \nu^{(i)}
=
 \delta_{x^{(i)}}
$
(Dirac's delta measure)
is
the initial distribution of $X^{(i)}$.
By
Le Gall's moment formula \eqref{eq_LGMF},
we have

\noindent
\begin{align}
\notag
&
 \mathbb{E}
 [
  |
   \langle
              {\ell^{\mathrm{IS}}_{t, \varepsilon}}
   ,
    f
   \rangle
  -
   \langle
    \ell^{\mathrm{IS}}_{t}
   ,
    f
   \rangle
  |^k
 ]
\\
\notag
=&
 \int_{{\E}^k}
  f(x_1) \cdots f(x_k)
  \prod_{i=1}^p
  \biggl(
   \sum_{\sigma\in \mathfrak{S}_k}
   (T_\varepsilon  - \mathrm{id})^{\otimes k}
   (H^{(i)}_t)
   (x_{\sigma(1)}, \ldots, x_{\sigma(k)})
  \biggr)
 \dm{x_1}\cdots \dm{x_k}
\\
\label{eq_superexp_3}
\leq&
 \|f\|_\infty^k
 (k!)^p
 \prod_{i=1} ^p
 \bigl\|
  (T_\varepsilon  - \mathrm{id})^{\otimes k}
  H^{(i)}_t
 \bigl\|_{L^p(\E^k)}
.
\end{align}

Fix $\delta > 0$.
We decompose $H^{(i)}_t$ as
\begin{equation*}
 H^{(i)}_t(x_1, \ldots, x_k)
=
 \sum_{\A \subset \{1, \ldots, k\}}
 H^{(i)}_t(\A; x_1, \ldots, x_k)
,
\end{equation*}

\noindent
where for $\A\subset \{1, \ldots, k\}$ we set
\begin{align}
\notag
&
 H^{(i)}_t(\A; x_1, \ldots, x_k)
\\
\label{eq_HtiA}
\begin{split}
 :=&
  \int_{\E}
  \int_{[0, \infty)^k}
   1_{\bigl\{\sum_{j=1} ^k \s_j \leq t\bigr\}}
    \prod_{j\in \A}
    1_{[0, \delta)}(\s_j)
    p_{\s_j}(x_{j-1}, x_j)
    1_{U}(x_j)
 \\
 &\hspace{39mm}
    \prod_{j\in \A^c}
    1_{[\delta, \infty)}(\s_j)
    p_{\s_j}(x_{j-1}, x_j)
    1_{U}(x_j)
  \hspace{2mm}
  d\s_1 \cdots d\s_k
  \nu^{(i)}(dx_0)
 .
\end{split}
\end{align}

\noindent
When $\#A$ is large,
we obtain the contribution of $\Cb$
from the indices in $A$
in the above function $H^{(i)}_t(A; \cdot)$.
On the other hand,
when $\#A$ is small,
we obtain the contribution of $\Ca$ from
(some subset of) $A^c$.
More precisely,
we have the following proposition:

\begin{Prop}
\label{Prop_estimate}
Let
$k\geq 3$ be an integer,
$\A\subset \{1, \ldots, k\}$
and $\delta > 0$.

\begin{enumerate}
\item[(i)]
When $\# \A > \frac{k}{4}$, it holds that
\begin{equation*}
 \bigl\|
  (T_\varepsilon  - \mathrm{id})^{\otimes k}
  H^{(i)}_t (\A; \cdot)
 \bigr\|_{L^p({\E}^k)}
\leq
 e^t
 2^k
 (\Cc+1)^k
 \Cb^{\frac{k}{4}}
.
\end{equation*}

\item[(ii)]
When $\# \A \leq \frac{k}{4}$, it holds that
\begin{equation*}
 \bigl\|
  (T_\varepsilon  - \mathrm{id})^{\otimes k}
  H^{(i)}_t (\A; \cdot)
 \bigr\|_{L^p({\E}^k)}
\leq
 e^t
 2^k
 (\Cc+1)^k
 \Ca^{\frac{k}{6}}
.
\end{equation*}
\end{enumerate}
\end{Prop}

We postpone
the proof Proposition \ref{Prop_estimate}
to the next section
and
complete the proof of Theorem \ref{Thm_superexp} first.
We have for $k\geq 3$,
\begin{align*}
&
 \bigl\|
  (T_\varepsilon  - \mathrm{id})^{\otimes k}
  H^{(i)}_t
 \bigr\|_{L^p({\E}^k)}
\\
\leq&
 \sum_{
  \substack{
   \A\subset \{1, \ldots, k\}
  ,
  \\
   \#\A> \frac{k}{4}
  }
 }
 \bigl\|
  (T_\varepsilon  - \mathrm{id})^{\otimes k}
  H^{(i)}_t(\A; \cdot)
 \bigr\|_{L^p({\E}^k)}
+
 \sum_{
  \substack{
   \A\subset \{1, \ldots, k\}
  ,
  \\
   \#\A\leq \frac{k}{4}
  }
 }
 \bigl\|
  (T_\varepsilon  - \mathrm{id})^{\otimes k}
  H^{(i)}_t(\A; \cdot)
 \bigr\|_{L^p({\E}^k)}
\\
\leq&
 e^t
 4^k
 (\Cc + 1)^k
 \{
  \Cb^{\frac{k}{4}}
 +
  \Ca^{\frac{k}{6}}
 \}
\\
\leq&
 e^t
 4^{k}
 (\Cc + 1)^k
 2
 (\Cb + \Ca)^{\frac{k}{6}}
.
\end{align*}

\noindent
Here we used
$
 \Cb + \Ca
<
 1
$.
By
combining this with \eqref{eq_superexp_3},
we have
for any even integer $k\geq 4$,

\noindent
\begin{align*}
 \mathbb{E}
 [
  |
   \langle
              {\ell^{\mathrm{IS}}_{t, \varepsilon}}
   ,
    f
   \rangle
  -
   \langle
    \ell^{\mathrm{IS}}_{t}
   ,
    f
   \rangle
  |^k
 ]
\leq&
 \|f\|_\infty^k
 (k!)^p
 \prod_{i=1} ^p
 \bigl\|
  (T_\varepsilon  - \mathrm{id})^{\otimes k}
  H^{(i)}_t
 \bigl\|_{L^p(\E^k)}
\\
\leq&
 e^{pt}
 (k!)^p
 \|f\|_\infty ^k
 \bigl\{
  8
  (\Cc + 1)
  (\Cb + \Ca)^{\frac{1}{6}}
 \bigr\}^{pk}
.
\end{align*}

\noindent
This proves \eqref{eq_superexp_2}
for even integer $k\geq 4$.
By applying Jensen's inequality
to the above inequality,
we have the desired bounds \eqref{eq_superexp_2}
for $k=1, 2$.
%
For any odd integer $k\geq 3$,
by combining Jensen's inequality
with the above inequalities for $k+1$,
we have

\noindent
\begin{align*}
 \mathbb{E}
 [
  |
   \langle
    {\ell^{\mathrm{IS}}_{t, \varepsilon}}
   ,
    f
   \rangle
  -
   \langle
    \ell^{\mathrm{IS}}_{t}
   ,
    f
   \rangle
  |^k
 ]
\leq&
 \mathbb{E}
 [
  |
   \langle
    {\ell^{\mathrm{IS}}_{t, \varepsilon}}
   ,
    f
   \rangle
  -
   \langle
    \ell^{\mathrm{IS}}_{t}
   ,
    f
   \rangle
  |^{k+1}
 ]^{\frac{k}{k+1}}
\\
\leq&
 \Bigl(
  e^{pt}
  ((k+1)!)^p
  \|f\|_\infty ^{k+1}
  \bigl\{
   8
   (\Cc + 1)
   (\Cb + \Ca)^{\frac{1}{6}}
  \bigr\}^{p(k+1)}
 \Bigr)^{\frac{k}{k+1}}
\\
\leq&
 e^{pt}
 (k!)^p
 \|f\|_\infty ^k
 \bigl\{
  16
  (\Cc + 1)
  (\Cb + \Ca)^{\frac{1}{6}}
 \bigr\}^{pk}
.
\end{align*}

\noindent
This proves \eqref{eq_superexp_2}
and thus
we complete the proof of Theorem \ref{Thm_superexp}.
\end{proof}

\begin{Rmk}
\label{Rmk_list}
The above proof uses the following three conditions:
\begin{align*}
 \lim_{\varepsilon\rightarrow 0}
 \Ca
=
 0
,
\quad
 \lim_{\delta\rightarrow 0}
 \Cb
=
 0
,
\quad
 \Cc < \infty
.
\end{align*}

\noindent
Theorem \ref{Thm_superexp} holds
not only for Brownian motion
but for processes that satisfy these conditions.
As
a representative example,
we discuss the stable process
in Section \ref{Sec_stable}.
\end{Rmk}


\subsection{Proof of Proposition \ref{Prop_estimate} (i)}
\label{Sec_estimate_1}

In this section,
we prove Proposition \ref{Prop_estimate} (i).
As
we stated in the previous section,
we will obtain the contribution $\Cb$
from the indices in $A$.
Fix
$\A\subset \{1, \ldots, k\}$ with $\#\A > \frac{k}{4}$.
We have
\begin{align*}
&
 H^{(i)}_t(\A; x_1, \ldots, x_k)
\\
\begin{split}
 =&
  \int_{\E}
  \int_{[0, \infty)^k}
   1_{\bigl\{\sum_{j=1} ^k \s_j \leq t\bigr\}}
    \prod_{j\in \A}
    1_{[0, \delta)}(\s_j)
    p_{\s_j}(x_{j-1}, x_j)
    1_{U}(x_j)
 \\
 &\hspace{39mm}
    \prod_{j\in \A^c}
    1_{[\delta, \infty)}(\s_j)
    p_{\s_j}(x_{j-1}, x_j)
    1_{U}(x_j)
  \hspace{2mm}
  d\s_1 \cdots d\s_k
  \nu^{(i)}(dx_0)
\end{split}
\\
\begin{split}
 \leq&
  e^t
  \int_{\E}
   \prod_{j\in \A}
   \int_0 ^\delta
    p_{\s_j}(x_{j-1}, x_j)
   d\s_j
   \prod_{j\in \A^c}
   r_1(x_{j-1}, x_j)
  \nu^{(i)}(dx_0)
 .
\end{split}
\end{align*}

\noindent
Recall the notations
$\Cb$ defined in \eqref{eq_tau}
and $\Cc$ in \eqref{eq_gamma}.
Then we have
\begin{align*}
 \bigl\|
  H^{(i)}_t(\A; \cdot)
 \bigl\|_{L^p(\E^k)}
\leq
 e^t
 \Cb^{\# \A}
 \Cc^{\# \A^c}
\leq
 e^t
 \Cb^{\frac{k}{4}}
 (\Cc+1)^{k}
\end{align*}

\noindent
and hence,
by combining this with the $L^p$-contractivity of
the operator $T_\varepsilon$,
we have
\begin{align*}
 \bigl\|
  (T_\varepsilon - \mathrm{id})^{\otimes k}
  [H^{(i)}_t(\A; \cdot)]
 \bigl\|_{L^p(\E^k)}
\leq
 2^k
 e^t
 \Cb^{\frac{k}{4}}
 (\Cc+1)^{k}
,
\end{align*}
which completes the proof.


\subsection{Proof of Proposition \ref{Prop_estimate} (ii); in case of $k=3$, $\A=\varnothing$}
\label{Sec_estimate_2}

In proving Proposition \ref{Prop_estimate} (ii),
we first deal with a simple case;
$k=3$ and $\A=\varnothing$.
The argument in this case contains a key
estimate that we will be used in the general case.

Until the end of the next section,
we simply write multiple integral
$
 \int_{[0, \infty)^k}
 d\s_1 \cdots d\s_k
$
as
$
 \int_{[0, \infty)^k}
 d\boldsymbol{\s}
$,
and
$
 \int_{{\Ebr}^p}
 \prod_{l=1} ^p dz^{(l)}
$
as
$
 \int_{{\Ebr}^{p}}
 d\boldsymbol{z}
$.
Recall the definition of
$H^{(i)}_t(\varnothing; x_1, x_2, x_3)$
in \eqref{eq_HtiA}.
By the change of variables, we have

\noindent
\begin{align*}
&
 H^{(i)}_t(\varnothing; x_1, x_2, x_3)
\\
 \begin{split}
 =&
  \int_{\E}
  \nu^{(i)}(dx_0)
  \int_{[0, \infty)^3}
  d\boldsymbol{\s}
 \hspace{2mm}
  1_{\bigl\{\s_1 + \s_2 + \s_3 \leq t-3\delta \bigr\}}
 \\
 &\hspace{5mm}
   \Bigl(
    p_{\s_1 + \delta}(x_{0}, x_1)
    1_{U}(x_1)
   \Bigr)
   \Bigl(
    p_{\s_2 + \delta}(x_{1}, x_2)
    1_{U}(x_2)
   \Bigr)
   \Bigl(
    p_{\s_3 + \delta}(x_{2}, x_3)
    1_{U}(x_3)
   \Bigr)
 \end{split}
\end{align*}

\noindent
and then, the Chapman-Kolmogorov equations
\begin{align*}
 p_{\s_1 + \delta}(x_0, x_1)
=&
 \int_{\E}
  p_{\frac{\delta}{2}+\s_1}(x_0, z_1)
  p_{\frac{\delta}{2}}(z_1, x_1)
 \dm{z_1}
,
\\
 p_{\s_2 + \delta}(x_{1}, x_2)
=&
 \int_{\E}
 \int_{\E}
  p_{\frac{\delta}{2}}(x_{1}, y_2)
  p_{\s_2}(y_2, z_2)
  p_{\frac{\delta}{2}}(z_2, x_2)
 \dm{y_2}
 \dm{z_2}
,
\\
 p_{\s_3 + \delta}(x_2, x_3)
=&
 \int_{\E}
  p_{\frac{\delta}{2}    }(x_2, y_3)
  p_{\frac{\delta}{2}+\s_3}(y_3, x_3)
 \dm{y_3}
\end{align*}
give that

\noindent
\begin{align*}
&
 H^{(i)}_t(\varnothing; x_1, x_2, x_3)
\\
 \begin{split}
 =&
  \int_{\E}
  \nu^{(i)}(dx_0)
  \int_{[0, \infty)^3}
  d\boldsymbol{\s}
  \int_{{\Ebr}^2}
  dy_2 dy_3
  \int_{{\Ebr}^2}
  dz_1 dz_2
  \hspace{2mm}
   1_{\bigl\{\s_1 + \s_2 + \s_3 \leq t-3\delta \bigr\}}
   p_{\frac{\delta}{2}+\s_1}(x_{0}, z_1)
  \\
  &\hspace{5mm}
   \Bigl(
    p_{\frac{\delta}{2}}(z_1, x_1)
    1_{U}(x_1)
    p_{\frac{\delta}{2}}(x_1, y_2)
   \Bigr)
   p_{\s_2}(y_2, z_2)
   \Bigl(
    p_{\frac{\delta}{2}}(z_2, x_2)
    1_{U}(x_2)
    p_{\frac{\delta}{2}}(x_2, y_3)
   \Bigr)
   \Bigl(
    p_{\frac{\delta}{2}+\s_3}(y_3, x_3)
    1_{U}(x_3)
   \Bigr)
 .
 \end{split}
\end{align*}

\noindent
%
As we mentioned
at the beginning of Section \ref{Sec_superexp},
the point is that
the integrand of the above equality
is separated as the functions of $x_1$, $x_2$ and $x_3$.
%
We will obtain the contribution of $\Ca$
from the functions of $x_1$ and $x_2$
by applying the operator
$(T_\varepsilon - \mathrm{id})$.
On the other hand,
the function of $x_3$ does not contribute
to the super-exponential estimate but it
is bounded from above by $\Cc$,
which is independent of $\varepsilon$.
For this reason, we first apply
$
 (T_\varepsilon - \mathrm{id})
 \otimes
 (T_\varepsilon - \mathrm{id})
 \otimes
 \mathrm{id}
$
to $H^{(i)}_t(\varnothing; \cdot)$ to get

\noindent
\begin{align*}
&
 (T_\varepsilon - \mathrm{id})
 \otimes
 (T_\varepsilon - \mathrm{id})
 \otimes
 \mathrm{id}
 [H^{(i)}_t(\varnothing; \cdot)](x_1, x_2, x_3)
\\
 \begin{split}
 =&
  \int_{\E}
  \nu^{(i)}(dx_0)
  \int_{[0, \infty)^3}
  d\boldsymbol{\s}
  \int_{{\Ebr}^2}
  dy_2 dy_3
  \int_{{\Ebr}^2}
  dz_1 dz_2
  \hspace{2mm}
   1_{\bigl\{\s_1 + \s_2 + \s_3 \leq t-3\delta \bigr\}}
   p_{\frac{\delta}{2}+\s_1}(x_{0}, z_1)
  \\
  &\hspace{5mm}
   \Bigl(
    (T_\varepsilon  - \mathrm{id})
    \bigl[
     p_{\frac{\delta}{2}}(z_1, \cdot)
     1_{U}(\cdot)
     p_{\frac{\delta}{2}}(\cdot, y_{2})
    \bigr]
    (x_1)
   \Bigr)
   p_{\s_2}(y_{2}, z_{2})
 \\
 &\hspace{10mm}
   \Bigl(
    (T_\varepsilon  - \mathrm{id})
    \bigl[
     p_{\frac{\delta}{2}}(z_{2}, \cdot)
     1_{U}(\cdot)
     p_{\frac{\delta}{2}}(\cdot, y_{3})
    \bigr]
    (x_2)
   \Bigr)
   \Bigl(
    p_{\frac{\delta}{2}+\s_3}(y_{3}, x_{3})
    1_{U}(x_{3})
   \Bigr)
 \end{split}
\end{align*}

\noindent
and then bound it as
\begin{align}
\notag
&
 \bigl|
  (T_\varepsilon - \mathrm{id})
  \otimes
  (T_\varepsilon - \mathrm{id})
  \otimes
  \mathrm{id}
  [H^{(i)}_t(\varnothing; \cdot)](x_1, x_2, x_3)
 \bigr|
\\
 \label{eq_resbdd}
 \begin{split}
 \leq&
  e^t
  \int_{\E}
  \nu^{(i)}(dx_0)
  \int_{{\Ebr}^2}
  dy_2 dy_3
  \int_{{\Ebr}^2}
  dz_1 dz_2
  \hspace{2mm}
   r_1(x_{0}, z_1)
  \\
  &\hspace{10mm}
   \Bigl|
    (T_\varepsilon  - \mathrm{id})
    \bigl[
     p_{\frac{\delta}{2}}(z_1, \cdot)
     1_{U}(\cdot)
     p_{\frac{\delta}{2}}(\cdot, y_{2})
    \bigr]
    (x_1)
   \Bigr|
   r_{1}(y_{2}, z_{2})
  \\
  &\hspace{15mm}
   \Bigl|
    (T_\varepsilon  - \mathrm{id})
    \bigl[
     p_{\frac{\delta}{2}}(z_{2}, \cdot)
     1_{U}(\cdot)
     p_{\frac{\delta}{2}}(\cdot, y_{3}))
    \bigr]
    (x_2)
   \Bigr|
   r_1(y_{3}, x_{3})
 .
 \end{split}
\end{align}


\noindent
By taking $p$-th power
and integrating $(x_1, x_2, x_3)$ over $\Ebr^3$,
we have
\begin{align}
\notag
&
 \bigl\|
  (T_\varepsilon - \mathrm{id})
  \otimes
  (T_\varepsilon - \mathrm{id})
  \otimes
  \mathrm{id}
  [H^{(i)}_t(\varnothing; \cdot)]
 \bigl\|_{L^p(\E^3)} ^p
\\
\label{eq_HLpE3_1}
 \begin{split}
 \leq&
  e^{pt}
  \int_{{\Ebr}^p}
  \nu^{(i)}(d\boldsymbol{x_0})
  \int_{{\Ebr}^{2p}}
  d\boldsymbol{y_2}
  d\boldsymbol{y_3}
  \int_{{\Ebr}^{2p}}
  d\boldsymbol{z_1}
  d\boldsymbol{z_2}
  \hspace{2mm}
   \prod_{l=1}^p
   r_1(x^{(l)}_{0}, z^{(l)}_1)
  \\
  &\hspace{10mm}
   \biggl(
    \int_{\E}
     \prod_{l=1}^p
     \Bigl|
      (T_\varepsilon  - \mathrm{id})
      \bigl[
       p_{\frac{\delta}{2}}(z^{(l)}_1, \cdot)
       1_{U}(\cdot)
       p_{\frac{\delta}{2}}(\cdot, y^{(l)}_{2})
      \bigr]
      (x_1)
     \Bigr|
    \dm{x_1}
   \biggl)
   \prod_{l=1}^p
   r_{1}(y^{(l)}_{2}, z^{(l)}_{2})
 \\
 &\hspace{15mm}
   \biggl(
    \int_{\E}
    \prod_{l=1}^p
     \Bigl|
      (T_\varepsilon  - \mathrm{id})
      \bigl[
       p_{\frac{\delta}{2}}(z^{(l)}_{2}, \cdot      )
       1_{U}(\cdot)
       p_{\frac{\delta}{2}}(\cdot      , y^{(l)}_{3})
      \bigr]
      (x_2)
     \Bigr|
    \dm{x_2}
   \biggl)
   \biggl(
    \int_{\E}
     \prod_{l=1}^p
     r_1(y^{(l)}_{3}, x_3)
    \dm{x_3}
   \biggr)
 .
 \end{split}
\end{align}

\noindent
We apply H\"older's inequality to
$
 \bigl(
  \int_{\E}
   \prod_{l=1}^p
   r_1(y^{(l)}_{3}, x_3)
  \dm{x_3}
 \bigr)
$
and recall the notation $\Cc$
introduced in \eqref{eq_gamma}
to bound \eqref{eq_HLpE3_1} by

\noindent
\begin{align}
\label{eq_HLpE3_2}
 \begin{split}
 &
  e^{pt} {\Cc}^p
  \int_{{\Ebr}^p}
  \nu^{(i)}(d\boldsymbol{x_0})
  \int_{{\Ebr}^{2p}}
  d\boldsymbol{y_2}
  d\boldsymbol{y_3}
  \int_{{\Ebr}^{2p}}
  d\boldsymbol{z_1}
  d\boldsymbol{z_2}
  \hspace{2mm}
   \prod_{l=1}^p
   r_1(x^{(l)}_{0}, z^{(l)}_1)
  \\
  &\hspace{15mm}
   \biggl(
    \int_{\E}
     \prod_{l=1}^p
     \Bigl|
      (T_\varepsilon  - \mathrm{id})
      \bigl[
       p_{\frac{\delta}{2}}(z^{(l)}_1, \cdot)
       1_{U}(\cdot)
       p_{\frac{\delta}{2}}(\cdot, y^{(l)}_{2})
      \bigr]
      (x_1)
     \Bigr|
    \dm{x}
   \biggl)
   \prod_{l=1}^p
   r_{1}(y^{(l)}_{2}, z^{(l)}_{2})
 \\
 &\hspace{20mm}
   \biggl(
    \int_{\E}
    \prod_{l=1}^p
     \Bigl|
      (T_\varepsilon  - \mathrm{id})
      \bigl[
       p_{\frac{\delta}{2}}(z^{(l)}_{2}, \cdot      )
       1_{U}(\cdot)
       p_{\frac{\delta}{2}}(\cdot      , y^{(l)}_{3})
      \bigr]
      (x_2)
     \Bigr|
    \dm{x_2}
   \biggl)
 .
 \end{split}
\end{align}

\noindent
We estimate the
integral of \eqref{eq_HLpE3_2} with respect to
$
 d\boldsymbol{y_3}
 d\boldsymbol{z_2}
$.
%
Regarding the integral with respect to
$
 d\boldsymbol{y_3}
$,
we apply H\"older's inequality to
$
 \bigl(
  \int_{\E}
  \prod_{l=1}^p
   \bigl|
    (T_\varepsilon  - \mathrm{id})
    \bigl[
     p_{\frac{\delta}{2}}(z^{(l)}_{2}, \cdot      )
     1_{U}(\cdot)
     \allowbreak
     p_{\frac{\delta}{2}}(\cdot      , y^{(l)}_{3})
    \bigr]
    (x_2)
   \bigr|
  \dm{x_2}
 \bigl)
$
and recall the notation $\Ca$
introduced in \eqref{eq_Cepdel}.
Regarding the integral with respect to
$
 d\boldsymbol{z_2}
$,
we use the trivial inequality
$
 \int_{\E}
  r_{1}(y^{(l)}_{2}, z^{(l)}_{2})
 \dm{z^{(l)}_{2}}
\leq
 1
$
for all $y^{(l)}_{2}$.
Then \eqref{eq_HLpE3_2} is bounded from above by

\noindent
\begin{align}
\label{eq_HLpE3_3}
 \begin{split}
 &
  e^{pt}
  {\Cc}^p
  {\Ca}^p
  \int_{{\Ebr}^p}
  \nu^{(i)}(d\boldsymbol{x_0})
  \int_{{\Ebr}^{p}}
  d\boldsymbol{y_2}
  \int_{{\Ebr}^{p}}
  d\boldsymbol{z_1}
  \hspace{2mm}
 \\
 &\hspace{5mm}
   \prod_{l=1}^p
   r_1(x^{(l)}_{0}, z^{(l)}_1)
   \biggl(
    \int_{\E}
     \prod_{l=1}^p
     \Bigl|
      (T_\varepsilon  - \mathrm{id})
      \bigl[
       p_{\frac{\delta}{2}}(z^{(l)}_1, \cdot)
       1_{U}(\cdot)
       p_{\frac{\delta}{2}}(\cdot, y^{(l)}_{2})
      \bigr]
      (x_1)
     \Bigr|
    \dm{x}
   \biggl)
 .
 \end{split}
\end{align}

\noindent
For the integral of \eqref{eq_HLpE3_3}
with respect to
$
 d\boldsymbol{y_2}
 d\boldsymbol{z_1}
$,
we repeat the argument for \eqref{eq_HLpE3_2}.
Eventually, we have

\noindent
\begin{align}
\notag
 \bigl\|
  (T_\varepsilon - \mathrm{id})
  \otimes
  (T_\varepsilon - \mathrm{id})
  \otimes
  \mathrm{id}
  [H^{(i)}_t(\varnothing; \cdot)]
 \bigl\|_{L^p(\E^3)} ^p
\leq&
 e^{pt}
 {\Cc}^p
 {\Ca}^{2p}
\end{align}

\noindent
and hence,
by the $L^p$-contractivity of the operator
$T_\varepsilon$ we conclude
\begin{align}
\notag
 \bigl\|
  (T_\varepsilon  - \mathrm{id})^{\otimes 3}
  H^{(i)}_t (\varnothing; \cdot)
 \bigl\|_{L^p(\E^3)}
\leq&
 2
 \bigl\|
  (T_\varepsilon  - \mathrm{id})
  \otimes
  (T_\varepsilon  - \mathrm{id})
  \otimes
  \mathrm{id}
  \bigl[
   H^{(i)}_t (\A; \cdot)
  \bigr]
 \bigl\|_{L^p(\E^3)}
\\
\notag
\leq&
 2
 e^{t}
 \Cc
 \Ca^{2}
\\
\label{eq_HLpE3_5}
\leq&
 2^{k}
 e^{t}
 (\Cc+1)^{k}
 \Ca^{\frac{1}{2}}
,
\end{align}
where in the last inequality,
recall that we take small $\varepsilon$ and $\delta$
so that $\Ca<1$.
Therefore
we complete the proof of Proposition
\ref{Prop_estimate} (ii) in this case.


\subsection{Proof of Proposition \ref{Prop_estimate} (ii); general case}
\label{Sec_estimate_3}

Now we prove
Proposition \ref{Prop_estimate} in general case.
Fix
$\A\subset \{1, \ldots, k\}$ with
$\#\A \leq \frac{k}{4}$.
We decompose $A^c$
into the following four disjoint parts
$
 \A^c
=
 F_1
\cup
 F_2
\cup
 F_3
\cup
 F_4
$:
\begin{align*}
 F_1
&:=
 \{
  1\leq j\leq k
 :
  j-1\not\in \A^c, j\in \A^c, j+1 \not\in \A^c
 \}
,
\\
 F_2
&:=
 \{
  1\leq j\leq k
 :
  j-1\not\in \A^c, j\in \A^c, j+1 \in \A^c
 \}
,
\\
 F_3
&:=
 \{
  1\leq j\leq k
 :
  j-1\in \A^c, j\in \A^c, j+1 \in \A^c
 \}
,
\\
 F_4
&:=
 \{
  1\leq j\leq k
 :
  j-1\in \A^c, j\in \A^c, j+1 \not\in \A^c
 \}
.
\end{align*}

\noindent
For example,
if $\A^c = \{1, 3, 4, 6, 7, 8, 9\}$,
then
$F_1 = \{1\}$,
$F_2 = \{3, 6\}$,
$F_3 = \{7, 8\}$,
and
$F_4 = \{4, 9\}$.
The previous section \S \ref{Sec_estimate_2}
is the case of
$A = F_1 = \varnothing$,
$F_2 = \{1\}$,
$F_3 = \{2\}$ and
$F_4 = \{3\}$.
%
%
The indices in $F_1$ and $F_4$ do not contribute
to the super-exponential estimate,
and the corresponding factors are bounded from above
by $\Cc$.
On the other hand,
from each index in $F_2$ and $F_3$,
we obtain the contribution of $\Ca$
as in the previous section.

We repeat the argument
in the previous section \S \ref{Sec_estimate_2}.
Recall the definition of
$H^{(i)}_t(A; x_1, x_2, x_3)$ in \eqref{eq_HtiA}.
By the change of variables, we have

\noindent
\begin{align*}
&
 H^{(i)}_t(\A; x_1, \ldots, x_k)
\\
 \begin{split}
 =&
  \int_{\E}
  \nu^{(i)}(dx_0)
  \int_{[0, \infty)^k}
  d\boldsymbol{\s}
  \hspace{2mm}
   1_{
    \bigl\{
     \sum_{j=1} ^k
     \s_j
    \leq
     t-\delta (\#F_2+\#F_3+\#F_4)
    \bigr\}
   }
   \prod_{j\in \A}
   1_{[0, \delta)}(\s_j)
   p_{\s_j}(x_{j-1}, x_j)
   1_{U}(x_j)
  \\
  &\hspace{5mm}
   \prod_{j\in F_1}
   1_{[\delta, \infty)}(\s_j)
   p_{\s_j}(x_{j-1}, x_j)
   1_{U}(x_j)
   \prod_{j\in F_2\cup F_3\cup F_4}
   p_{\s_j+\delta}(x_{j-1}, x_j)
   1_{U}(x_j)
 \end{split}
\end{align*}

\noindent
and then, the Chapman-Kolmogorov equations
\begin{align*}
 p_{\s_j + \delta}(x_{j-1}, x_j)
=&
 \int_{\E}
  p_{\frac{\delta}{2}+\s_j}(x_{j-1}, z_j)
  p_{\frac{\delta}{2}}(z_j, x_j)
 \dm{z_j}
&
 \hspace{-15mm}
 \text{for }j\in F_2
,
\\
 p_{\s_j + \delta}(x_{j-1}, x_j)
=&
 \int_{\E}
 \int_{\E}
  p_{\frac{\delta}{2}}(x_{j-1}, y_j)
  p_{\s_j}(y_j, z_j)
  p_{\frac{\delta}{2}}(z_j, x_j)
 \dm{y_j}
 \dm{z_j}
&
 \hspace{-15mm}
 \text{for }j\in F_3
,
\\
 p_{\s_j + \delta}(x_{j-1}, x_j)
=&
 \int_{\E}
  p_{\frac{\delta}{2}    }(x_{j-1}, y_j)
  p_{\frac{\delta}{2}+\s_j}(y_j, x_j)
 \dm{y_3}
&
 \hspace{-15mm}
 \text{for }j\in F_4
\end{align*}
give that

\noindent
\begin{align*}
&
 H^{(i)}_t(\A; x_1, \ldots, x_k)
\\
 \begin{split}
 =&
  \int_{\E}
  \nu^{(i)}(dx_0)
  \int_{[0, \infty)^k}
  d\boldsymbol{\s}
  \int_{{\Ebr}^{F_3\cup F_4}}
  d\boldsymbol{y}
  \int_{{\Ebr}^{F_2\cup F_3}}
  d\boldsymbol{z}
  \hspace{2mm}
   1_{
    \bigl\{
     \sum_{j=1} ^k
     \s_j
    \leq
     t-\delta (\#F_2+\#F_3+\#F_4)
    \bigr\}
   }
  \\
  &\hspace{5mm}
   \prod_{j\in \A}
   1_{[0, \delta)}(\s_j)
   p_{\s_j}(x_{j-1}, x_j)
   1_{U}(x_j)
   \prod_{j\in F_1}
   1_{[\delta, \infty)}(\s_j)
   p_{\s_j}(x_{j-1}, x_j)
   1_{U}(x_j)
  \\
  &\hspace{5mm}
   \prod_{j\in F_2}
   p_{\frac{\delta}{2}+\s_j}(x_{j-1}, z_j)
   \Bigl(
    p_{\frac{\delta}{2}}(z_{j}, x_{j}  )
    1_{U}(x_{j})
    p_{\frac{\delta}{2}}(x_{j}, y_{j+1})
   \Bigr)
  \\
  &\hspace{5mm}
   \prod_{j\in F_3}
   p_{\s_j             }(y_j    , z_j)
   \Bigl(
    p_{\frac{\delta}{2}}(z_{j}, x_{j}  )
    1_{U}(x_{j})
    p_{\frac{\delta}{2}}(x_{j}, y_{j+1})
   \Bigr)
  \\
  &\hspace{5mm}
   \prod_{j\in F_4}
   p_{\s_j+\frac{\delta}{2}}(y_j    , x_j)
   1_{U}(x_j)
 .
 \end{split}
\end{align*}

\noindent
%
Here again, the point is that
the integrand of the above equality is
separated as the functions of
$(x_j)_{j\in A\cup F_1}$
and
$x_j$, $j\in F_2\cup F_3\cup F_4$
because of the Chapman-Kolmogorov equations.
As we mentioned at the beginning of this section,
we obtain the contribution $\Ca$
by applying the operator
$(T_\varepsilon -\mathrm{id})$
to
each function of the indices in $F_2$ and $F_3$.
On the other hand,
the indices $A$, $F_1$ and $F_4$ do not
contribute to the super-exponential estimate,
and the factor is bounded above by $C_3$,
which is independent of $\varepsilon$.
Note that
$
 F_2 \cup F_3
\subset
 \{1, \ldots, k-1\}
$.
By setting
\begin{equation*}
 U_j
:=
 \begin{cases}
  (T_\varepsilon - \mathrm{id})
 &
  \text{when }
  j \in F_2 \cup F_3
 ,
 \\
  \mathrm{id}
 &
  \text{otherwise},
 \end{cases}
\end{equation*}

\noindent
we have
\begin{align*}
&
 (U_1\otimes \cdots \otimes U_k)
 [H^{(i)}_t(\A; \cdot)]
 (x_1, \ldots, x_k)
\\
 \begin{split}
 =&
  \int_{\E}
  \nu^{(i)}(dx_0)
  \int_{[0, \infty)^k}
  d\boldsymbol{\s}
  \int_{{\Ebr}^{F_3\cup F_4}}
  d\boldsymbol{y}
  \int_{{\Ebr}^{F_2\cup F_3}}
  d\boldsymbol{z}
  \hspace{2mm}
   1_{
    \bigl\{
     \sum_{j=1} ^k
     \s_j
    \leq
     t-\delta (\#F_2+\#F_3+\#F_4)
    \bigr\}
   }
  \\
  &\hspace{5mm}
   \prod_{j\in \A}
   1_{[0, \delta)}(\s_j)
   p_{\s_j}(x_{j-1}, x_j)
   1_{U}(x_j)
   \prod_{j\in F_1}
   1_{[\delta, \infty)}(\s_j)
   p_{\s_j}(x_{j-1}, x_j)
   1_{U}(x_j)
  \\
  &\hspace{5mm}
   \prod_{j\in F_2}
   p_{\frac{\delta}{2}+\s_j}(x_{j-1}, z_j)
   \Bigl(
    (T_\varepsilon - \mathrm{id})
    \bigl[
    p_{\frac{\delta}{2}}(z_{j}, \cdot  )
    p_{\frac{\delta}{2}}(\cdot, y_{j+1})
    1_{U}(\cdot)
    \bigr]
    (x_{j})
   \Bigr)
  \\
  &\hspace{5mm}
   \prod_{j\in F_3}
   p_{\s_j             }(y_j    , z_j)
   \Bigl(
    (T_\varepsilon - \mathrm{id})
    \bigl[
    p_{\frac{\delta}{2}}(z_{j}, \cdot  )
    p_{\frac{\delta}{2}}(\cdot, y_{j+1})
    1_{U}(\cdot)
    \bigr]
    (x_{j})
   \Bigr)
  \\
  &\hspace{5mm}
   \prod_{j\in F_4}
   p_{\s_j+\frac{\delta}{2}}(y_j    , x_j)
   1_{U}(x_j)
 \end{split}
\end{align*}

\noindent
and then,
since
$
 \sum_{j=1}^k
 \s_j
+
 \frac{\delta}{2}
 (\# F_2 + \# F_4)
\leq
 t
$,
we have
\begin{align*}
&
 \bigl|
  (U_1\otimes \cdots \otimes U_k)
  [H^{(i)}_t(\A; \cdot)]
  (x_1, \ldots, x_k)
 \bigr|
\\
 \begin{split}
 \leq&
  e^t
  \int_{\E}
  \nu^{(i)}(dx_0)
  \int_{{\Ebr}^{F_3\cup F_4}}
  d\boldsymbol{y}
  \int_{{\Ebr}^{F_2\cup F_3}}
  d\boldsymbol{z}
   \prod_{j\in \A}
   r_1(x_{j-1}, x_j)
   \prod_{j\in F_1}
   r_1(x_{j-1}, x_j)
  \\
  &\hspace{5mm}
   \prod_{j\in F_2}
   r_1(x_{j-1}, z_j)
   \Bigl|
    (T_\varepsilon - \mathrm{id})
    \bigl[
    p_{\frac{\delta}{2}}(z_{j}, \cdot  )
    p_{\frac{\delta}{2}}(\cdot, y_{j+1})
    1_{U}(\cdot)
    \bigr]
    (x_{j})
   \Bigr|
  \\
  &\hspace{5mm}
   \prod_{j\in F_3}
   r_1(y_j, z_j)
   \Bigl|
    (T_\varepsilon - \mathrm{id})
    \bigl[
    p_{\frac{\delta}{2}}(z_{j}, \cdot  )
    p_{\frac{\delta}{2}}(\cdot, y_{j+1})
    1_{U}(\cdot)
    \bigr]
    (x_{j})
   \Bigr|
   \prod_{j\in F_4}
   r_1(y_j, x_j)
 .
 \end{split}
\end{align*}

\noindent
We repeat the argument
from \eqref{eq_resbdd} to \eqref{eq_HLpE3_5}.
Then
we have
\begin{align}
\notag
 \bigl\|
  (U_1\otimes \cdots \otimes U_k)
  [H^{(i)}_t(\A; \cdot)]
 \bigl\|_{L^p(\E^k)}
\leq&
 e^t
 {\Cc}^{\#\A + \#F_1 + \#F_4}
 {\Ca}^{\# F_2 + \# F_3}
\\
\label{eq_HLpEk}
\leq&
 e^t
 (\Cc+1)^k
 \Ca^{\frac{k}{6}}
,
\end{align}

\noindent
where we used $\Ca<1$.
In
the second line \eqref{eq_HLpEk},
we used the estimate
$
 \#F_2 + \#F_3
\geq
 \frac{k}{6}
$
which is obtained as follows:
the minimum of
$
 \# F_2 + \# F_3
$
over $\# \A \leq \frac{k}{4}$
is attained for
$
 \A
=
 \{
  2l
 :
  1\leq l \leq \# \A
 \}
$,
and in this case,
we have
$
 \# F_2 + \# F_3
=
 \#
 \{
  2\#\A + 1
 ,
  2\#\A + 2
 ,
  \ldots
 ,
  k-1
 \}
=
 k - 1 - 2\# \A
\geq
 \frac{k}{6}
$
since $k\geq 3$.

%

Set
\begin{equation*}
 V_j
:=
 \begin{cases}
  \mathrm{id}
 &
  \text{when }
  j \in F_2 \cup F_3
 ,
 \\
  (T_\varepsilon - \mathrm{id})
 &
  \text{otherwise}.
 \end{cases}
\end{equation*}

\noindent
By combining \eqref{eq_HLpEk} with
the $L^p$-contractivity of
the operator $T_\varepsilon$,
we have
\begin{align*}
 \bigl\|
  (T_\varepsilon - \mathrm{id})^{\otimes k}
  [H^{(i)}_t(\A; \cdot)]
 \bigl\|_{L^p(\E^k)}
=&
 \bigl\|
  (V_1\otimes \cdots \otimes V_k)
  (U_1\otimes \cdots \otimes U_k)
  [H^{(i)}_t(\A; \cdot)]
 \bigl\|_{L^p(\E^k)}
\\
\leq&
 2^k
 \bigl\|
  (U_1\otimes \cdots \otimes U_k)
  [H^{(i)}_t(\A; \cdot)]
 \bigl\|_{L^p(\E^k)}
\\
\leq&
 2^k
 e^t
 (\Cc+1)^k
 \Ca^{\frac{k}{6}}
,
\end{align*}
which concludes the proof of
Proposition \ref{Prop_estimate} (ii).


\subsection{Estimate of $\Ca$}
\label{Sec_Cepsdel}

In this section,
we estimate the constant introduced in
\eqref{eq_Cepdel}:

\noindent
\begin{equation*}
 \Ca
:=
 \sup_{z\in \E}
 \biggl\{
  \int_{\E}
   \Bigl\|
    (T_\varepsilon - \mathrm{id})
    \bigl[
    p_{\frac{\delta}{2}}(z    , \cdot)
    p_{\frac{\delta}{2}}(\cdot, y    )
    1_{U}(\cdot)
    \bigr]
   \Bigr\|_{L^p(\E)}
  \dm{y}
 \biggr\}
.
\end{equation*}

\begin{Lem}
It holds that
\begin{equation*}
 \lim_{\varepsilon \rightarrow 0}
 \Ca
=
 0
\quad
 \text{for every }
 \delta > 0
.
\end{equation*}
\end{Lem}

\begin{proof}
\label{Lem_Cepsdel}
We easily find that
the transition density function $p_t(x, y)$
of a killed Brownian motion in $D$
has the following properties
(see \cite[Theorem 2.4]{MR1329992} for example):

\noindent
\begin{equation}
\label{eq_Lem_Cepsdel_1}
 p_{\frac{\delta}{2}}(\cdot, \cdot)
\text{ is continuous on }
 \E\times \E
,
\end{equation}

\noindent
\begin{equation}
\label{eq_Lem_Cepsdel_2}
 \int_{\E}
  \bigl\|
   p_{\frac{\delta}{2}}(\cdot, y)
   1_{U}(\cdot)
  \bigr\|_{L^p(\E)}
 \dm{y}
<
 \infty
,
\end{equation}

\noindent
\begin{equation}
\label{eq_Lem_Cepsdel_3}
 \lim_{\substack{z\in D,\\ z\rightarrow \partial}}
 p_{\frac{\delta}{2}}(z, x)
=
 0
\quad
 \text{for each }x\in D
.
\end{equation}

Now,
we can see that
\eqref{eq_Lem_Cepsdel_1}
implies
$
 \lim_{\varepsilon \rightarrow 0}
 \bigl\|
  (T_\varepsilon - \mathrm{id})
  \bigl[
   p_{\frac{\delta}{2}}(z    , \cdot)
   p_{\frac{\delta}{2}}(\cdot, y    )
   1_{U}(\cdot)
  \bigr]
 \bigr\|_{L^p(\E)}
=
 0
$
for each $y, z\in \E$.
By taking
a relatively compact neighborhood $V$ of $z$,
we
can also see the bound

\noindent
\begin{equation}
\label{eq_Lem_Cepsdel_4}
 \bigl\|
  (T_\varepsilon - \mathrm{id})
  \bigl[
   p_{\frac{\delta}{2}}(z    , \cdot)
   p_{\frac{\delta}{2}}(\cdot, y    )
   1_{U}(\cdot)
  \bigr]
 \bigr\|_{L^p(\E)}
\leq
 2
 \biggl(
  \sup_{V\times U}
  p_{\frac{\delta}{2}}(\cdot, \cdot)
 \biggr)
 \bigl\|
  \bigl[
   p_{\frac{\delta}{2}}(\cdot, y    )
   1_{U}(\cdot)
  \bigr]
 \bigr\|_{L^p(\E)}
\end{equation}
for each $\varepsilon > 0$ and each $y\in \E$.
Note that
the upper bound in \eqref{eq_Lem_Cepsdel_4}
is an integrable function of $y$
because of \eqref{eq_Lem_Cepsdel_1}
and \eqref{eq_Lem_Cepsdel_2}.
By the dominated convergence theorem, we have

\noindent
\begin{equation}
\label{eq_Lem_Cepsdel_5}
 \lim_{\varepsilon \rightarrow 0}
 \int_{\E}
  \bigl\|
   (T_\varepsilon - \mathrm{id})
   \bigl[
   p_{\frac{\delta}{2}}(z    , \cdot)
   p_{\frac{\delta}{2}}(\cdot, y    )
   1_{U}(\cdot)
   \bigr]
  \bigr\|_{L^p(\E)}
 \dm{y}
=
 0
\quad
 \text{for fixed }z\in \E
\end{equation}
and find that


\noindent
\begin{equation}
\label{eq_Lem_Cepsdel_6}
 \E\ni z
\longmapsto
 \int_{\E}
  \bigl\|
   (T_\varepsilon - \mathrm{id})
   \bigl[
    p_{\frac{\delta}{2}}(z    , \cdot)
    p_{\frac{\delta}{2}}(\cdot, y    )
    1_{U}(\cdot)
   \bigr]
  \bigl\|_{L^p(\E)}
 \dm{y}
\quad
 \text{is continuous}
\end{equation}
for each $\varepsilon > 0$.

Next,
\eqref{eq_Lem_Cepsdel_2}
and
\eqref{eq_Lem_Cepsdel_3}
give the estimate

\noindent
\begin{align}
\notag
&
 \sup_{\varepsilon > 0}
 \int_{\E}
  \bigl\|
   (T_\varepsilon - \mathrm{id})
   \bigl[
    p_{\frac{\delta}{2}}(z    , \cdot)
    p_{\frac{\delta}{2}}(\cdot, y    )
    1_{U}(\cdot)
   \bigr]
  \bigr\|_{L^p(\E)}
 \dm{y}
\\
\label{eq_Lem_Cepsdel_7}
\leq&
 2
 \biggl(
  \sup_{x\in U}
  p_{\frac{\delta}{2}}(z, x)
 \biggr)
 \int_{\E}
  \bigl\|
   p_{\frac{\delta}{2}}(\cdot, y)
   1_{U}(\cdot)
  \bigr\|_{L^p(\E)}
 \dm{y}
\end{align}
and the right-hand side goes to $0$
as $z\rightarrow \partial$.
Therefore,
the desired uniform convergence
$
 \lim_{\varepsilon\rightarrow 0}
 \Ca
=
 0
$
follows from
\eqref{eq_Lem_Cepsdel_5},
\eqref{eq_Lem_Cepsdel_6}
and
\eqref{eq_Lem_Cepsdel_7}.
\end{proof}


\section{Large deviation lower bound}
\label{Sec_lowerLDP}

In this section we prove Theorem \ref{Thm_LDP} (i),
the LDP lower bound.
Let $X$ be a killed Brownian motion
in a domain $D\subset \mathbb{R}^d$
with smooth boundary.
Define
the occupation measure $\ell_t$ of $X$ up to $t$ by
$
 \langle f, \ell_t \rangle
=
 \int_0 ^t
  f(X_s)
 ds
$
for bounded Borel functions $f$ on $D$.
We
first recall the well known Donsker-Varadhan type
large deviation lower bound
for the normalized occupation measure $t^{-1}\ell_t$
on $(\mathcal{M}_1(D), \tau_w)$.

\begin{Thm}
[{\cite[Proposition 4.1]{MR1666887}}]
\label{Thm_LDPocp}
Define the function
$
 I: \mathcal{M}_1(D)\longrightarrow [0, +\infty]
$
by
\begin{equation*}
 I(\mu)
=
\begin{cases}
 \displaystyle
 \frac{1}{2}
 \int_D
  |\nabla \psi|^2
 dx
&
 \text{if }
 \mu = \psi^2 dx
,
 \psi \in W^{1, 2}_0(D)
,
 \psi\geq 0
,
\\
 \infty
&
 \text{otherwise}
\end{cases}
\end{equation*}

\noindent
for $\mu\in \mathcal{M}_1(D)$.
Then,
on the space
$
 (\mathcal{M}_1(D), \tau_w)
$,
the family of occupation measures
$\{t^{-1}\ell_t\}_t$
satisfies the LDP lower bound
as $t\rightarrow \infty$
under
$
 \mathbb{P}
 (
  \hspace{1mm}\cdot\hspace{1mm}
 ,
  t< \tau_D
 )
$
with the rate function $I$.
\end{Thm}


For each $\varepsilon >0$,
we define the function
$
 \Phi_\varepsilon
:
 (\mathcal{M}_1(D), \tau_w)^p
\longrightarrow
 (\mathcal{M}(D), \tau_v)
\times
 (\mathcal{M}_{1}(D), \tau_w)^p
$
by
\begin{equation*}
 \Phi_\varepsilon
 (
  \mu^{(1)}, \ldots, \mu^{(p)}
 )
:=
 \biggl(
  \biggl[
   \prod_{i=1} ^p
   q_\varepsilon[\mu^{(i)}](x)
  \biggr]
  dx
 \hspace{1mm};\hspace{1mm}
  \mu^{(1)}, \ldots, \mu^{(p)}
 \biggr)
.
\end{equation*}
Since the function $\Phi_\varepsilon$ is continuous,
the contraction principle of LDP gives the following.

\begin{Lem}
Define the function
$
 \mathbf{I}_\varepsilon
:
 \mathcal{M}(D) \times (\mathcal{M}_{1}(D))^p
\longrightarrow
 [0, +\infty]
$
by
\begin{align*}
&
 \mathbf{I}_\varepsilon
 (
  \mu
 ;
  \mu^{(1)}, \ldots, \mu^{(p)}
 )
\\
&\hspace{5mm}
:=
 \inf
 \biggl\{
  \sum_{i=1} ^p
   I (\nu^{(i)})
 \hspace{1mm}
 \biggl|
 \hspace{1mm}
  \nu^{(1)}, \ldots, \nu^{(p)}
 \in
  \mathcal{M}_1(D)
 ,
  \Phi_\varepsilon
  (
   \nu^{(1)}, \ldots, \nu^{(p)}
  )
 =
  (
   \mu
  ;
   \mu^{(1)}, \ldots, \mu^{(p)}
  )
 \biggr\}
\\
&\hspace{5mm}
=
 \begin{cases}
  \displaystyle
  \frac{1}{2}
  \sum_{i=1} ^p
  \int_D
   |\nabla \psi^{(i)}|^2
  dx
 ,
 &
  \text{if }
  \psi^{(i)}
 =
  \displaystyle
  \sqrt{\frac{d\mu^{(i)}}{dx}}
 \in
  W^{1, 2}_0(D)
 \text{ and }
  \displaystyle
  \prod_{i=1} ^p
  q_{\varepsilon} [\mu^{(i)}]
 =
  \frac{d\mu}{dx}
 ,
 \\
  \infty
 ,
 &
  \text{otherwise}
 \end{cases}
\end{align*}

\noindent
for
$
 \left(
  \mu
 ;
  \mu^{(1)}
 ,
  \ldots
 ,
  \mu^{(p)}
 \right)
\in
 \mathcal{M}(D)
\times
 (\mathcal{M}_{1} (D) )^p
$.

Then,
for any open set
$
 G
\subset
 (\mathcal{M}(D), \tau_v)
\times
 (\mathcal{M}_{1}(D), \tau_w)^p
$,
it holds that
\begin{equation*}
 \liminf_{t\rightarrow\infty}
 \frac{1}{t}
 \log
 \mathbb{P}
 \Bigl(
 (
  t^{-p}\ell^{\mathrm{IS}}_{t, \varepsilon}
 ;
  t^{-1}\ell^{(1)}_t
 ,
  \ldots
 ,
  t^{-1}\ell^{(p)}_t
 )
  \in
 G
 ,
  t
 <
  \tau^{(1)}_D
 \wedge
  \cdots
 \wedge
  \tau^{(p)}_D
 \Bigr)
\geq
 -\inf_{\boldsymbol{\mu}\in G}
 \mathbf{I}_\varepsilon
 (\boldsymbol{\mu})
.
\end{equation*}
\end{Lem}

\noindent
We will show that $\mathbf{I}$
defined in \eqref{eq_bfI} is a rate function
later,
so we don't check
whether $\mathbf{I}_\varepsilon$ is a rate function
or not.


We next give the relation
between
$\mathbf{I}$
and
so-called the $\Gamma$-lower limit of
$\mathbf{I}_\varepsilon$.

\begin{Prop}
\label{Prop_appbfI}
For every
$
 \boldsymbol{\mu}
\in
 \mathcal{M}(D)
\times
 (\mathcal{M}_{1}(D) )^p
$,
it holds that
\begin{equation*}
 \mathbf{I}
 (\boldsymbol{\mu})
\geq
 \sup_{\delta>0}
 \liminf_{\varepsilon \downarrow 0}
 \inf_{
  \boldsymbol{\nu}
 \in
  \mathbf{B}_\delta
  (\boldsymbol{\mu})
 }
 \mathbf{I}_\varepsilon
 (\boldsymbol{\nu})
,
\end{equation*}
where
$
 \mathbf{B}_\delta
 (\boldsymbol{\mu})
$
is the open ball
with center $\boldsymbol{\mu}$ and radius $\delta$,
with respect to a metric of
$
 (\mathcal{M}(D), \tau_v)
\times
 (\mathcal{M}_{1}(D), \tau_w )^p
$.
\end{Prop}

\begin{proof}
The following is based on the proof of
\cite[Proposition 1.2]{MR2999298}.
Let
$
 \boldsymbol{\mu}
=
 (\mu; \mu^{(1)}, \ldots, \mu^{(p)})
\in
 \mathcal{M}(D)
\times
 (\mathcal{M}_{1}(D))^p
$
with
$
 \mathbf{I}
 (\boldsymbol{\mu})
<
 \infty
$
be given.
Take
nonnegative $\psi^{(i)} \in W^{1, 2}_0(D)$
such that
$
 \mu^{(i)}(dx)
=
 (\psi^{(i)})^2
 dx
$
and
$
 \mu(dx)
=
 \bigl[
  \prod_{i=1} ^p
  (\psi^{(i)})^2
 \bigr]
 dx
$.

Fix
$\delta > 0$ and take $\varepsilon >0$ so small
such that
$
 \bigl[
  \prod_{i=1} ^p
  q_\varepsilon [\mu^{(i)}]
 \bigr]
 dx
\in
 B_{\delta/2p} (\mu)
$.
This is possible,
since
%
the triangle inequality,
H\"older's inequality
and the $L^p(\mathbb{R}^d)$-contractivity of $q_\varepsilon$
give that
\begin{align}
\notag
&
 \biggl\|
  \prod_{i=1} ^p
  q_\varepsilon [(\psi^{(i)} )^2]
 -
  \prod_{i=1} ^p
  (\psi^{(i)} )^2
 \biggr\|_{L^1(D)}
\\
\notag
 \begin{split}
 \leq&
  \biggl\|
   \prod_{i=1} ^p
   q_\varepsilon [(\psi^{(i)} )^2 ]
  -
   (\psi^{(1)} )^2
   \prod_{i=2} ^p
   q_\varepsilon [(\psi^{(i)} )^2 ]
  \biggr\|_{L^1(D)}
 \\
 &\hspace{15mm}
 +
  \biggl\|
   (\psi^{(1)} )^2
   \prod_{i=2} ^p
   q_\varepsilon [(\psi^{(i)} )^2 ]
  -
   (\psi^{(1)} )^2
   (\psi^{(2)} )^2
   \prod_{i=3} ^p
   q_\varepsilon [(\psi^{(i)} )^2 ]
  \biggr\|_{L^1(D)}
 \\
 &\hspace{30mm}
 +
  \cdots
 +
  \biggl\|
   \biggl(
    \prod_{i=1} ^{p-1}
    (\psi^{(i)} )^2
   \biggr)
   q_\varepsilon [(\psi^{(p)} )^2 ]
  -
   \prod_{i=1} ^p
   (\psi^{(i)} )^2
  \biggr\|_{L^1(D)}
 \end{split}
\\
\notag
\leq&
 \sum_{i=1} ^p
 \biggl(
  \prod_{l< i}
  \|
   \psi^{(l)}
  \|_{L^{2p}(D)} ^2
 \biggr)
 \|
  q_\varepsilon [(\psi^{(i)} )^2]
 -
  (\psi^{(i)} )^2
 \|_{L^{p}(\mathbb{R}^d)}
 \biggl(
  \prod_{l>i}
  \|
   q_\varepsilon [(\psi^{(l)} )^2]
  \|_{L^{p}(\mathbb{R}^d)}
 \biggr)
\\
\label{eq_appbfJ_2}
\leq&
 \sum_{i=1} ^p
 \|
  q_\varepsilon [(\psi^{(i)} )^2]
 -
  (\psi^{(i)} )^2
 \|_{L^{p}(\mathbb{R}^d)}
 \biggl(
  \prod_{l\not=i}
  \|
   \psi^{(l)}
  \|_{L^{2p}(D)} ^2
 \biggr)
.
\end{align}
The last line \eqref{eq_appbfJ_2}
goes to $0$ as $\varepsilon \rightarrow 0$,
because of
the Sobolev embedding theorem
$
 W^{1, 2}_0(D)
\subset
 L^{2p}(D)
$
(recall the assumption $d-p(d-2)>0$,
i.e., $2p< 2d/(d-2)$)
and
the $L^p(\mathbb{R}^d)$-continuity of $q_\varepsilon$.
Hence
we have for any $f\in C_K(D)$
\begin{align*}
 \Biggl|
  \biggl\langle
   f
  ,
   \prod_{i=1} ^p
   q_\varepsilon [\mu^{(i)}]
  \biggr\rangle
 -
  \langle f, \mu \rangle
 \Biggr|
=&
 \Biggl|
  \biggl\langle
   f
  ,
   \prod_{i=1} ^p
   q_\varepsilon [(\psi^{(i)} )^2 ]
  \biggr\rangle
 -
  \biggl\langle
   f
  ,
   \prod_{i=1} ^p
   (\psi^{(i)} )^2
  \biggr\rangle
 \Biggr|
\\
\leq&
 \|f\|_\infty
 \biggl\|
  \prod_{i=1} ^p
  q_\varepsilon [(\psi^{(i)} )^2 ]
 -
  \prod_{i=1} ^p
  (\psi^{(i)} )^2
 \biggr\|_{L^{1}(D)}
\end{align*}

\noindent
and the right-hand side
goes to $0$ as $\varepsilon \rightarrow 0$.
We thus obtain
$
 \bigl(
  \bigl[
   \prod_{i=1} ^p
   {q_\varepsilon [\mu^{(i)}]}
  \bigr]
  dx
 ;
  \hspace{1mm}
 \mu^{(1)}
 ,\linebreak
  \ldots
 ,
  \mu^{(p)}
 \bigr)
\in
 \mathbf{B}_\delta (\boldsymbol{\mu})
$
and hence

\noindent
\begin{align*}
 \inf_{\mathbf{B}_\delta(\boldsymbol{\mu})}
 \textbf{I}_\varepsilon
\leq&
 \textbf{I}_\varepsilon
 \biggl(
  \biggl[
   \prod_{i=1} ^p
   {q_\varepsilon [\mu^{(i)}]}
  \biggr]
  dx
 ;
  \hspace{1mm}
  \mu^{(1)}
 ,
  \ldots
 ,
  \mu^{(p)}
 \biggr)
\leq
 \mathbf{I}(\boldsymbol{\mu})
,
\end{align*}
which concludes the proof.
\end{proof}


\begin{proof}[Proof of Theorem \ref{Thm_LDP} (i)]
We first
prove that $\mathbf{I}$ is a rate function.
Let
$\alpha > 0$ be fixed.
Suppose
a sequence
$
 \{
  (
   \mu_n;, \mu^{(1)}_n, \ldots, \mu^{(p)}_n
  )
 \}
\subset
 \{\mathbf{I} \leq \alpha\}
$
and take nonnegative $\psi^{(i)}_n\in W^{1,2}_0(D)$
such that
$
 \mu^{(i)}_n(dx)
=
 (\psi^{(i)}_n)^2
 dx
$
and
$
 \mu_n(dx)
=
 \bigl[
  \prod_{i=1} ^p
  (\psi^{(i)}_n)^2
 \bigr]
 dx
$.
We assume that
$\mu_n$ converges to $\mu$
in $(\mathcal{M}(D), \tau_v)$ and
$\mu^{(i)}_n$ converges to $\mu^{(i)}$
in $(\mathcal{M}_1(D), \tau_w)$.
Since
$\{\psi^{(i)}\}$ is bounded in $W^{1, 2}_0(D)$
for each $i$,
by taking a subsequence we may assume that
$\psi^{(i)}_n$ converges weakly to
$\psi^{(i)}\in W^{1, 2}_0(D)$ for all $i$.

For $f\in C_K(D)$,
take a bounded open set $U\subset D$
with smooth boundary such that
$\mathrm{supp}[f]\subset U$ and
$\overline{U}\subset D$.
The
Rellich-Kondrashov theorem gives that
$\{\psi^{(i)}_n 1_U\}$ converges strongly to
$\psi^{(i)} 1_U$ in $W^{1, 2}(U)$ for each $i$
(it hold that $2p<2d/(d-2)$ as we wrote
below \eqref{eq_appbfJ_2}).
Then
we have

\noindent
\begin{align}
\label{eq_upper_1}
 \int_D
  f (\psi^{(i)})^2
 dx
=
 \int_U
  f (\psi^{(i)} 1_U)^2
 dx
=
 \lim_{n\rightarrow \infty}
 \int_U
  f (\psi^{(i)}_n 1_U)^2
 dx
=
 \lim_{n\rightarrow \infty}
 \langle
  f, \mu^{(i)}_n
 \rangle
\end{align}

\noindent
and have
\begin{align}
\label{eq_upper_2}
 \frac{1}{2}
 \sum_{i=1} ^p
 \int_{D}
  f
  |\nabla \psi^{(i)}|^2
 dx
=&
 \frac{1}{2}
 \sum_{i=1} ^p
 \int_{U}
  f
  |\nabla (\psi^{(i)} 1_U)|^2
 dx
=
 \lim_{n\rightarrow \infty}
 \frac{1}{2}
 \sum_{i=1} ^p
 \int_{U}
  f
  |\nabla (\psi^{(i)}_n 1_U)|^2
 dx
\leq
 \|f\|_\infty \alpha
.
\end{align}
By the same way as to obtain \eqref{eq_appbfJ_2},
we also have

\noindent
\begin{align*}
&
 \biggl\|
  \prod_{i=1} ^p
  (\psi^{(i)}_n 1_U)^2
 -
  \prod_{i=1} ^p
  (\psi^{(i)}  1_U)^2
 \biggr\|_{L^1(\E)}
\\
\notag
\leq&
 \sum_{i=1} ^p
 \|
  (\psi^{(i)}_n 1_U)^2
 -
  (\psi^{(i)}   1_U)^2
 \|_{L^{p}(\E)}
 \prod_{l\not= i}
 \biggl(
  \sup_n
  \|
   \psi^{(l)}_n
  \|_{L^{2p}(\E)} ^2
 +
  \|
   \psi^{(l)}
  \|_{L^{2p}(D)} ^2
 \biggr)
.
\end{align*}

\noindent
The right-hand side of the above inequality
goes to $0$ as $\varepsilon \rightarrow 0$,
because of
the Sobolev and Rellich-Kondrashov embedding theorem.
Hence
$\mu_n$
converges to
$
 \allowbreak
 \bigl[
  \prod_{i=1} ^p
  (\psi^{(i)})^2)
 \bigr]
 dx
$
in
$
 (\mathcal{M}(D), \tau_v)
$
and therefore
$
 \mu(dx)
=
 \bigl[
  \prod_{i=1} ^p
  (\psi^{(i)})^2)
 \bigr]
 dx
$,
$
 \mu^{(i)}(dx)
=
 (\psi^{(i)})^2 dx
$
and
$
 \mathbf{I}(\mu; \mu^{(1)}, \ldots, \mu^{(p)})
\leq
 \alpha
$.
Hence $\mathbf{I}$ is a rate function.

\vspace{1eM}

As we have seen
below Theorem \ref{Thm_superexp},
the tuple of random measures
$
 \{
  (
   t^{-p}\ell^{\mathrm{IS}}_{t, \varepsilon}
  ;
   t^{-1}\ell^{(1)}_t
  ,\allowbreak
   \ldots
  ,
   t^{-1}\ell^{(p)}_t
  )
 \}_{t, \varepsilon}
$
are exponentially good approximations of
$
 \{
  (
   t^{-p}\ell^{\mathrm{IS}}_{t}
  ;
   t^{-1}\ell^{(1)}_t
  ,
   \ldots
  ,
   t^{-1}\ell^{(p)}_t
  )
 \}_t
$.
Then
it is straightforward
to get the desired LDP lower bound
from Proposition \ref{Prop_appbfI}.
See, for example,
the lower bound of the proof of
\cite[Theorem 4.2.16 (a)]{MR1619036}.
%
%
\end{proof}


\section{Large deviation upper bound}
\label{Sec_upperLDP}

In this section we prove Theorem \ref{Thm_LDP} (ii),
the LDP upper bound.
We also prove Proposition \ref{Prop_LDPbdd}
at the end of this section.
As we mentioned in Section \ref{Sec_outline},
when the domain $D$ is unbounded,
we need to consider the case
that some mass of the (normalized) occupation measure
of a Brownian motion escapes to infinity.
Hence,
for the occupation measure,
it is natural to consider the full LDP on the space
$\mathcal{M}_1(D_\partial)$
even when $D=\mathbb{R}^d$.


Let
$\ell_t$ be the occupation measure of
a killed Brownian motion up to $t$
and
regard this as a measure on $D_\partial$.
Just
as in the previous section,
we
have the Donsker-Varadhan type
large deviation upper bound
for the normalized occupation measure $t^{-1}\ell_t$
on the compactified space
$(\mathcal{M}_1(D_\partial), \tau_w)$.

\begin{Lem}
\label{Lem_LDPpartial}
Define the function
$
 I^\partial
:
 \mathcal{M}_1(D_\partial)
\longrightarrow
 [0, +\infty]
$
by
\begin{equation*}
 I^\partial (\mu)
=
\begin{cases}
 \displaystyle
 \frac{1}{2}
 \int_D
  |\nabla \psi|^2
 dx
&
 \text{if }
 \mu = \psi^2 dx + c\delta_\partial
,
 \psi \in W^{1, 2}_0(D)
,
 \psi\geq 0
,
 c\geq 0
,
\\
 \infty
&
 \text{otherwise}
\end{cases}
\end{equation*}

\noindent
for $\mu\in \mathcal{M}_1(D_\partial)$.
Then,
on the space
$
 (\mathcal{M}_1(D_\partial), \tau_w)
$,
the family of occupation measures
$\{t^{-1}\ell_t\}_t$
satisfies the LDP upper bound
as $t\rightarrow \infty$
under
$
 \mathbb{P}
 (
  \hspace{1mm}\cdot\hspace{1mm}
 ,
  t< \tau_D
 )
$
with the good rate function $I^\partial$.
\end{Lem}

\begin{proof}
We first prove
the compactness of the level set.
Let
$
 \{\mu_n\}
\subset
 \{I^\partial \leq \alpha\}
$
and take $\psi_n\in W^{1,2}_0(D)$
such that $\mu_n|_D(dx) = \psi_n^2 dx$.
By
the same way as to obtain
\eqref{eq_upper_1} and \eqref{eq_upper_2},
there exists
$\psi\in W^{1, 2}_0(D)$ with $\|\psi\|_2 \leq 1$
such that,
by taking a subsequence,
$\mu_n$ converges to
$\mu:=\psi^2 dx + (1-\|\psi\|_2 ^2)\delta_\partial$
in $(\mathcal{M}_1(D_\partial), \tau_w)$
and
$I^\partial(\mu) \leq \alpha$.
Hence $I^\partial$ is a good rate function.

The LDP upper bound is proved in \cite{MR0486395}
in the case of $D=\mathbb{R}^d$,
More general case,
we use the same argument as that to obtain
\cite[(3.7)]{MR2360927}
or
\cite[(4.5)]{MR3514698}.
\end{proof}

\vspace{1eM}

For
each $\varepsilon >0$
we define the function
$
 \Phi^\partial_\varepsilon
:
 (\mathcal{M}_1(D_\partial), \tau_w)^p
\longrightarrow
 (\mathcal{M}(D), \tau_v)
\times
 (\mathcal{M}_{\leq 1}(D), \tau_v)^p
$
by

\noindent
\begin{equation*}
 \Phi^\partial_\varepsilon
 (
  \mu^{(1)}, \ldots, \mu^{(p)}
 )
:=
 \biggl(
  \biggl[
   \prod_{i=1} ^p
   q_\varepsilon[\mu^{(i)}|_D](x)
  \biggr]
  dx
 ;
  \mu^{(1)}|_D, \ldots, \mu^{(p)}|_D
 \biggr)
.
\end{equation*}

\noindent
Since the function $\Phi^\partial_\varepsilon$
is continuous,
the contraction principle of LDP gives the following.

\begin{Lem}
Define the function
$
 \overline{\mathbf{I}}_\varepsilon
:
 \mathcal{M}(D) \times \mathcal{M}_{\leq 1}(D)^p
\longrightarrow
 [0, +\infty]
$
by
\begin{align*}
&
 \overline{\mathbf{I}}_\varepsilon
 (
  \mu
 ;
  \mu^{(1)}, \ldots, \mu^{(p)}
 )
\\
:=&
 \inf
 \biggl\{
  \sum_{i=1} ^p
   I^\partial (\nu^{(i)})
 \hspace{1mm}
 \biggl|
 \hspace{1mm}
  \nu^{(1)}, \ldots, \nu^{(p)}
 \in
  \mathcal{M}_1(D_\partial)
 ,
  \Phi^\partial_\varepsilon
  (
   \nu^{(1)}, \ldots, \nu^{(p)}
  )
 =
  (
   \mu
  ;
   \mu^{(1)}, \ldots, \mu^{(p)}
  )
 \biggr\}
\\
=&
 \begin{cases}
  \displaystyle
  \frac{1}{2}
  \sum_{i=1} ^p
  \int_D
   |\nabla \psi^{(i)}|^2
  dx
 ,
 &
  \text{if }
  \psi^{(i)}
 =
  \displaystyle
  \sqrt{\frac{d\mu^{(i)}}{dx}}
 \in
  W^{1, 2}_0(D)
 \text{ and }
  \displaystyle
  \frac{d\mu}{dx}
 =
  \prod_{i=1} ^p
  q_{\varepsilon} [\mu^{(i)}]
 ,
 \\
  \infty
 ,
 &
  \text{otherwise}
 ,
 \end{cases}
\end{align*}

\noindent
for
$
 \left(
  \mu
 ;
  \mu^{(1)}, \ldots, \mu^{(p)}
 \right)
\in
 \mathcal{M}(D)
\times
 (\mathcal{M}_{\leq 1} (D) )^p
$.

Then,
on the space
$
 (\mathcal{M}(D), \tau_v)
\times
 (\mathcal{M}_{\leq 1}(D), \tau_v)^p
$,
the law of the tuple
$
 (
  t^{-p}\ell^{\mathrm{IS}}_{t, \varepsilon}
 ;
  t^{-1}\ell^{(1)}_t
 ,\allowbreak
  \ldots
 ,
  t^{-1}\ell^{(p)}_t
 )
$
satisfies the LDP upper bound
as $t\rightarrow \infty$
under
$
 \mathbb{P}
 (
  \hspace{1mm}\cdot\hspace{1mm}
 ,
  t
 <
  \tau^{(1)}_D
 \wedge
  \cdots
 \wedge
  \tau^{(p)}_D
 )
$,
with the good rate function
$
 \overline{\mathbf{I}}_\varepsilon
$.
\end{Lem}


Recall the function
$
 \overline{\mathbf{I}}
$
defined in Section \ref{Sec_outline}.

\begin{Prop}
\label{Prop_appsubbfI}
For every closed set
$
 F
\subset
 (\mathcal{M}(D), \tau_v)
\times
 (\mathcal{M}_{\leq 1}(D), \tau_v)^p
$,
it holds that
\begin{equation}
\label{eq_appbfI_2}
 \inf_{\boldsymbol{\mu}\in F}
 \overline{\mathbf{I}}
 (\boldsymbol{\mu})
\leq
 \liminf_{\varepsilon\rightarrow 0}
 \inf_{\boldsymbol{\mu}\in F}
 \overline{\mathbf{I}}_\varepsilon
 (\boldsymbol{\mu})
.
\end{equation}
\end{Prop}


\begin{proof}
Without loss of generality, we may assume that
\begin{equation*}
 R
:=
 \liminf_{\varepsilon \downarrow 0}
 \inf_{\boldsymbol{\mu}\in F}
 \overline{\mathbf{I}}_\varepsilon
 (\boldsymbol{\mu})
<
 \infty
.
\end{equation*}

\noindent
Fix $\eta > 0$.
Then for each $\varepsilon > 0$
we can pick
$
 \boldsymbol{\mu}_{\varepsilon}
=
 (
  \mu_{\varepsilon}
 ;
  \mu^{(1)}_{\varepsilon}
 ,
  \ldots
 ,
  \mu^{(p)}_{\varepsilon}
 )
\in
 F
$
with
$
 \overline{\mathbf{I}}_\varepsilon
 (
  \boldsymbol{\mu}_{\varepsilon}
 )
\leq
 R
+
 \eta
$.
By
the definition of $\overline{\mathbf{I}}_\varepsilon$,
there are nonnegative
$
 \psi^{(i)} _{\varepsilon}
\in
 W^{1, 2}_0(D)
$
such that
$
 \mu^{(i)} _{\varepsilon}(dx)
=
 (\psi^{(i)} _{\varepsilon} )^2
 dx
$
and
$
 \mu _{\varepsilon}(dx)
=
 \bigl[
  \prod_{i=1} ^p
  q_\varepsilon
  [(\psi^{(i)} _{\varepsilon} )^2]
 \bigr]
 dx
$.
In particular,
$
 \frac{1}{2}
 \sum_{i=1} ^p
 \int_D
  |\nabla \psi^{(i)}_{\varepsilon}|^2
 dx
\leq
 R+ \eta
$
and hence
$
 \{
  \psi^{(i)}_{\varepsilon}
 \}_\varepsilon
$
is bounded in $W^{1, 2}_0(D)$.

Set
$
 \mu^{\prime}_\varepsilon
=
 \bigl[
  \prod_{i=1} ^p
  (\psi^{(i)} _{\varepsilon} )^2
 \bigr]
 dx
$
and
$
 \boldsymbol{\mu}^{\prime}_\varepsilon
=
 (
  \mu^{\prime}_\varepsilon
 ;
  \mu^{(1)}_\varepsilon
 ,
  \ldots,
  \mu^{(p)}_\varepsilon
 )
$
.
By the same way
as the proof of Theorem \ref{Thm_LDP} (i),
for some sequence $\varepsilon_n\downarrow 0$
(in the following, we write this
as $\varepsilon\downarrow 0$
with some abuse of notations),
$
 \boldsymbol{\mu}^{\prime}_\varepsilon
$
converges to some
$
 \boldsymbol{\mu}
=
 (
  \mu
 ;
  \mu^{(1)}
 ,
  \ldots,
  \mu^{(p)}
 )
$
in
$
 (\mathcal{M}(D), \tau_v)
\times
 (\mathcal{M}_{\leq 1}(D), \tau_v)^p
$.
By the same way as
to obtain \eqref{eq_appbfJ_2}, we also have
for any a bounded open set $U \subset \E$
with smooth boundary and $\overline{U} \subset D$,

\noindent
\begin{align*}
&
 \biggl\|
  \prod_{i=1} ^p
  q_{\varepsilon}
  [(\psi^{(i)}_\varepsilon)^2]
 -
  \prod_{i=1} ^p
  (\psi^{(i)}_{\varepsilon})^2
 \biggr\|_{L^1(U)}
\leq
 \sum_{i=1} ^p
 \|
  q_{\varepsilon}
  [(\psi^{(i)}_\varepsilon)^2]
 -
  (\psi^{(i)}_{\varepsilon})^2
 \|_{L^p(U)}
 \prod_{l\not=i}
 \biggl(
  \sup_\varepsilon
  \|
   \psi^{(i)}_{\varepsilon}
  \|_{L^{2p}(D)}^2
 \biggr)
,
\end{align*}
which goes to $0$ as $\varepsilon \rightarrow 0$
because of
the Sobolev and Rellich-Kondrashov embedding theorems.
Hence
$
 \boldsymbol{\mu}_\varepsilon
$
converges to
$
 \boldsymbol{\mu}
$
in
$
 (\mathcal{M}(D), \tau_v)
\times
 (\mathcal{M}_{\leq 1}(D), \tau_v)^p
$.

Therefore, we have
$
 \boldsymbol{\mu}
\in
 F
$
and
$
 \inf_F
 \overline{\mathbf{I}}
\leq
 \overline{\mathbf{I}}
 (\boldsymbol{\mu})
\leq
 R + \eta
$.
The conclusion \eqref{eq_appbfI_2} follows
by letting $\eta\rightarrow 0$.
\end{proof}

\vspace{1eM}


\begin{proof}[Proof of Theorem \ref{Thm_LDP} (ii)]
We can show that
$\overline{\mathbf{I}}$ is a good rate function
by
a similar way to the proof of
Theorem \ref{Thm_LDP} (i).
As we have seen in below Theorem \ref{Thm_superexp},
$
 \{
  (
   t^{-p}\ell^{\mathrm{IS}}_{t, \varepsilon}
  ;
   t^{-1}\ell^{(1)}_t
  ,
   \ldots
  ,
   t^{-1}\ell^{(p)}_t
  )
 \}_{t, \varepsilon}
$
are exponentially good approximations of
$
 \{
  (
   t^{-p}\ell^{\mathrm{IS}}_{t}
  ;\allowbreak
   t^{-1}\ell^{(1)}_t
  ,
   \ldots
  ,
   t^{-1}\ell^{(p)}_t
  )
 \}_{t}
$.
By combining
this with Proposition \ref{Prop_appsubbfI},
it is straightforward to get the desired
LDP upper bound.
See, for example,
the lower bound of the proof of
\cite[Theorem 4.2.16 (b)]{MR1619036}.
%
\end{proof}

\begin{proof}[Proof of Proposition \ref{Prop_LDPbdd}]
When the domain $D$ is bounded,
the rate function $I$ defined in
Theorem \ref{Thm_LDPocp} is indeed a good rate
function and the upper LDP also holds for
the normalized occupation measure $t^{-1}\ell_t$
(see \cite[Theorem 1.1]{MR2851247} for example).
We
repeat the arguments in this section
with replacing
$I^\partial$,
$\overline{\mathbf{I}}_\varepsilon$ and
$\overline{\mathbf{I}}$
by
$I$,
${\mathbf{I}}_\varepsilon$ and
${\mathbf{I}}$, respectively.
Then the
desired LDP for the intersection measure
follows.
\end{proof}


\section{Large deviation principle for the intersection measure of stable processes}
\label{Sec_stable}

In this section,
we discuss the LDP for the intersection measure
of stable processes.
Throughout this section, let $\alpha\in (0, 2)$.
We consider
a rotationally symmetric $\alpha$-stable
process killed upon leaving a domain $D$
with smooth boundary.
It is known that
(see \cite{MR2677618} for example)
the process has a transition density function
$p_t(x, y)$
with respect to the Lebesgue measure
that is jointly continuous
and
$p_t(x, y)$ has a following upper estimate:
for
every $T>0$ there exists a constant $C>0$
such that

\noindent
\begin{align*}
 p_t(x, y)
\leq
 C
 \biggl(
  1
 \wedge
  \frac{\delta_D(x)^{\alpha/2}}{\sqrt{t}}
 \biggr)
 \biggl(
  1
 \wedge
  \frac{\delta_D(y)^{\alpha/2}}{\sqrt{t}}
 \biggr)
 \biggl(
  t^{-d/\alpha}
 \wedge
  \frac{t}{|x-y|^{d+\alpha}}
 \biggr)
\end{align*}
for all
$
 (t, x, y)\in (0, T]\times D\times D
$,
where
$\delta_D(x)$ is the Euclidean distance
between $x$ and $D^c$.
The following embedding theorem
and compact embedding theorem
for the fractional Sobolev space
$W^{\frac{\alpha}{2}, 2}(\mathbb{R}^d)$
are also known (see \cite{MR2944369} for example):

\begin{Thm}
Suppose $\alpha < d$.
The Banach space

\noindent
\begin{equation*}
 W^{\frac{\alpha}{2}, 2}(\mathbb{R}^d)
:=
 \biggl\{
  u\in L^2(\mathbb{R}^d)
 :
  \int_{\mathbb{R}^d}
  \int_{\mathbb{R}^d}
   \frac{|u(x)-u(y)|^2}{|x-y|^{d+\alpha}}
  dx
  dy
 <
  \infty
 \biggr\}
\end{equation*}

\noindent
equipped with the norm
\begin{equation*}
 \|u\|_{W^{\frac{\alpha}{2}, 2}}^2
=
 \|u\|_{L^2}^2
+
 \int_{\mathbb{R}^d}
 \int_{\mathbb{R}^d}
  \frac{|u(x)-u(y)|^2}{|x-y|^{d+\alpha}}
 dx
 dy
,
\quad
 u\in W^{\frac{\alpha}{2}, 2}(\mathbb{R}^d)
\end{equation*}

\noindent
is continuously embedded in $L^{2p}(\mathbb{R}^d)$
for $p\geq 1$ with
$
 d -p(d-\alpha) > 0
$.

\end{Thm}

\begin{Thm}
Suppose $\alpha < d$
and
$U \subset \mathbb{R}^d$
be a bounded open set with smooth boundary.
If
$\mathcal{J}$
is a bounded subset of $L^2(U)$ satisfying
\begin{equation*}
 \sup_{u\in \mathcal{J}}
 \int_U
 \int_U
  \frac{|u(x) - u(y)|^2}{|x-y|^{d+\alpha}}
 dx
 dy
<
 \infty
,
\end{equation*}

\noindent
then
$\mathcal{J}$
is relatively compact in $L^{2p}(U)$
for $p\geq 1$ with
$
 d -p(d-\alpha) > 0
$.
\end{Thm}

Now suppose $\alpha < d$ and $d-p(d-\alpha)>0$.
We can find that
until the previous section
we used only the following conditions
as properties of a killed Brownian motion:

\begin{itemize}
\setlength{\itemsep}{0mm}
\item
$
 \lim_{\delta\downarrow 0}
 \Cb
=
 0
$
and
$
 \Cc<\infty
$
(recall \eqref{eq_tau} and \eqref{eq_gamma}
for notation),

\item
conditions
\eqref{eq_Lem_Cepsdel_1},
\eqref{eq_Lem_Cepsdel_2} and
\eqref{eq_Lem_Cepsdel_3}
(and hence
$
 \lim_{\varepsilon \rightarrow 0}
 \Ca
=
 0
$),

\item
the Sobolev and Rellich-Kondrashov embedding theorems.
\end{itemize}

\noindent
Then our main result
Theorem \ref{Thm_LDPunbdd}
also holds for
the intersection measure of killed stable processes
by replacing
$
 W^{1, 2}_0(D)
$
and
$
 \int_\E
  |\nabla \psi|^2
 dx
$
with
$
 W^{\frac{\alpha}{2}, 2}_0(D)
:=
 \overline{C_0 ^\infty(D)}^{
 W^{\frac{\alpha}{2}, 2}
 }
$
and
$
 \int_{\E}
 \int_{\E}
  \frac{|\psi(x)-\psi(y)|^2}{|x-y|^{d+\alpha}}
 dx
 dy
$,
respectively.
%

\section*{Acknowledgements}
\addcontentsline{toc}{section}{Acknowledgements}
The author would like to thank Professor Takashi Kumagai and
Professor Ryoki Fukushima for helpful discussions.
He is also grateful to Professor Chiranjib Mukherjee
for explaining him the content of \cite{MR3716856}.
This work
was supported by JSPS KAKENHI Grant Number JP18J21141.

\part*{}
\addcontentsline{toc}{section}{References}
\bibliographystyle{abbrvalpha}

\end{document}